\documentclass[reqno]{amsart}
\usepackage[scale=0.75, centering, headheight=14pt]{geometry}
\usepackage[T1]{fontenc}
\usepackage{lmodern}
\usepackage[english]{babel}

\usepackage{tikz}
\usepackage{bbm}
\usepackage{amsmath,amssymb,amsfonts,amsthm}
\usepackage{mathtools,accents}
\usepackage{mathrsfs}
\usepackage{xfrac}
\usepackage{array} 
\usepackage{aliascnt}
\usepackage{booktabs} 
\usepackage{array} 

\usepackage{verbatim} 
\usepackage{subfig} 

\usepackage{mathrsfs, dsfont}
\usepackage{amssymb}
\usepackage{amsthm}
\usepackage{amsmath,amsfonts,amssymb,esint}
\usepackage{graphics,color}
\usepackage{enumerate}
\usepackage{mathtools,centernot}

\usepackage{microtype}
\usepackage{paralist} 
\usepackage{cases}
\usepackage[initials]{amsrefs}
\allowdisplaybreaks

\usepackage{braket}
\usepackage{bm}

\usepackage[citecolor=blue,colorlinks]{hyperref}
\addto\extrasenglish{}

\usepackage{enumerate}
\usepackage{xcolor}
\usepackage{aliascnt}

\makeatletter
\def\newaliasedtheorem#1[#2]#3{
  \newaliascnt{#1@alt}{#2}
  \newtheorem{#1}[#1@alt]{#3}
  \expandafter\newcommand\csname #1@altname\endcsname{#3}
}
\makeatother

\def\avint{\mathop{\mathchoice{\,\rlap{-}\!\!\int}
                              {\rlap{\raise.15em{\scriptstyle -}}\kern-.2em\int}
                              {\rlap{\raise.09em{\scriptscriptstyle -}}\!\int}
                              {\rlap{-}\!\int}}\nolimits}

\def\avint{\mathop{\,\rlap{-}\!\!\int}\nolimits}

\newcommand{\A}{\mathcal{A}}

\DeclareMathOperator{\id}{id}

\newcommand{\R}{\mathbb{R}}
\newcommand{\E}{\mathds{E}}
\newcommand{\U}{\mathcal{U}}
\newcommand{\N}{\mathbb{N}}

\DeclareMathOperator{\curl}{curl}

\DeclareMathOperator{\m}{m}

\DeclareMathOperator{\cof}{cof}
\DeclareMathOperator{\rank}{rank}
\DeclareMathOperator{\Id}{Id}

\DeclareMathOperator{\dx}{dx}
\DeclareMathOperator{\dy}{dy}
\DeclareMathOperator{\dv}{div}
\DeclareMathOperator{\tr}{tr}

\DeclareMathOperator{\spt}{spt}

\DeclareMathOperator*{\intt}{int}

\DeclareMathOperator{\loc}{loc}
\DeclareMathOperator{\Sym}{Sym}

\DeclareMathOperator{\Lip}{Lip}

\DeclareMathOperator{\dist}{dist}

\theoremstyle{plain}
\newtheorem*{NTeo}{Theorem}

\newtheorem{Teo}{Theorem}[section]
\newtheorem{lemma}[Teo]{Lemma}
\newtheorem{prop}[Teo]{Proposition}

\newtheorem{Cor}[Teo]{Corollary}
\theoremstyle{definition}
\newtheorem{Def}[Teo]{Definition}
\theoremstyle{remark}
\newtheorem{rem}[Teo]{Remark}

\title{Minimal graphs and differential inclusions}

\author[Tione]{Riccardo Tione}
\address[Riccardo Tione]{
\newline \indent Institut f\"ur Mathematik, Universit\"at Z\"urich,
\newline \indent Winterthurerstrasse 190, CH-8057 Zurich, Switzerland}
\email{riccardo.tione@math.uzh.ch}

\begin{document}

\maketitle

\begin{abstract}
In this paper, we study the differential inclusion associated to the minimal surface system for two-dimensional graphs in $\R^{2 + n}$. We prove regularity of $W^{1,2}$ solutions and a compactness result for approximate solutions of this differential inclusion in $W^{1,p}$. Moreover, we make a perturbation argument to infer that for every $R > 0$ there exists $\alpha(R) >0$ such that $R$-Lipschitz stationary points for functionals $\alpha$-close in the $C^2$ norm to the area functional are always regular. We also use a counterexample of \cite{KIRK} to show the existence of irregular critical points to inner variations of the area functional.
\end{abstract}
\par
\medskip\noindent
{\sc MSC (2010): 35B20 - 35B65 - 49Q05 - 58E12.
\par
}

\section{Introduction}

The history of the study of stationary graphs for the area function (sometimes called minimal graphs) is extremely rich, see \cite[Chapter 6]{MGS2}, \cite[Chapter 11]{GM} and \cite{OSS} for a more geometric approach. Consider $\Omega$ convex open set and a Lipschitz function $u: \Omega \subset \R^k \to \R^n$, such that the graph $(x,u(x))$ is a stationary point for the area function in $\R^{k + n}$. In other words, $u$ solves the following system
\begin{equation}\label{sys0}
\begin{cases}
\displaystyle \sum_{i,j = 1}^k\frac{\partial}{\partial x^i}\left(\sqrt{g}g^{ij}\frac{\partial u^\ell}{\partial x^j}\right) = 0,\; &\ell \in \{1,\dots,n\};\\
\displaystyle \sum_{i = 1}^k\frac{\partial}{\partial x^i}\sqrt{g}g^{ij} = 0,\; &j \in \{1,\dots,k\},
\end{cases}
\end{equation}
where $(g^{ij}) = (g_{ij})^{-1}$, $g_{ij} = (D u_i,D u_j)$ and $g = \det((g_{ij}))$.
In the case $n = k + 1$, the work of many mathematicians such as F. Almgren, E. Bombieri, E. De Giorgi, E. Giusti, H. Jenkins, J. Serrin, J. Simons and others has given a fairly complete understanding of the problem (see \cite{GM,GT,MM}). In the case $n > k + 1$, the problem is more complicated and many properties of solutions of \eqref{sys0} in the codimension-one case fail in the higher-codimension case. In the seminal paper \cite{LOA}, L. B. Lawson and R. Osserman proved that in general the Dirichlet problem associated to \eqref{sys0} has no unique solution and that its solutions are not always stable. Concerning the regularity of such solutions, Lawson and Osserman showed that if $k\ge 4$ and $n + k \ge 7$, then singular Lipschitz stationary graphs exist. Some years later, D. Fischer-Colbrie proved in \cite{DFC} that if $k = 2,3$ then every Lipschitz stationary graph is smooth. For later developments of the theory, see also \cite{GRE,WANG,WAN1}.
\\
\\
\indent Let us consider any polyconvex (i.e. a convex function of the subminors of a matrix $X$) $f:\R^{n\times 2}\to \R$ of class $C^1$ and let us fix a convex, bounded and open set $\Omega \subset \R^2$. The equations of \eqref{sys0} are a particular case of the system that a stationary point for an energy of the form
\[
\E_f(u) \doteq \int_{\Omega}f(D u)\dx, \text{ for } u \in \Lip(\Omega,\R^n),
\]
solves. If $$f(X) = \sqrt{1 + \|X\|^2 + \sum_{1\le a\le b\le n}\det(X^{ab})^2},\; \forall X \in \R^{n\times 2},$$ then $\E_f(\cdot)$ measures the area of the graph $(x,u(x))$. In this case, we will denote $f(\cdot) = \A(\cdot)$, and we will call it the area function. We say that $u \in W^{1,p}$ (for $p \ge 2$) is a stationary point for $f$ if $u$ is a critical point with respect to both outer and inner variations, i.e. if $u$ solves
\begin{equation}\label{sys}
\begin{cases}
\displaystyle \int_{\Omega}\langle(D f)(D u), \; D v\rangle\dx= 0, &\forall v\in C^1_c(\Omega,\mathbb{R}^n),\vspace{0.1cm}\\ 
\displaystyle \int_{\Omega}\langle(D f)(D u), \; D u D \Phi(x)\rangle \dx - \int_{\Omega}f(D u)\dv(\Phi) \dx = 0, &\forall \Phi\in C_c^1(\Omega,\mathbb{R}^2).
\end{cases}
\end{equation}
It can be shown, see \cite{TN}, that the graphs of functions $u$ satisfying the previous system are actually stationary in the sense of varifolds. 
\\
\\
\indent In order to study the solutions of \eqref{sys}, we recast \eqref{sys} as a differential inclusion. First of all, we rewrite \eqref{sys} in its classical form
\begin{equation*}
\begin{cases}
\dv(Df(D u)) = 0,\\
\dv((D u)^TDf(D u) - f(D u)\id) = 0.
\end{cases}
\end{equation*}
Then, using Poincar\'e's Lemma, we infer that $u$ is a solution of system \eqref{sys} if and only if there exist functions $v: \Omega \to \R^n$ and $w:\Omega \to \R^2$ such that the function $$\U = \left(\begin{array}{c}u\\v\\w \end{array}\right): \Omega \to \R^{2n + 2 }$$ fulfills
\begin{equation}\label{diffincF}
D \U(x) \in C_f = \left\{Y\in\R^{(2n+ 2)\times 2}: Y=
\left(
\begin{array}{c}
X\\
Df(X)J\\
X^TDf(X)J-f(X)J
\end{array}
\right)
\right\}, \text{ for a.e. } x \in \Omega,
\end{equation}
where $J \in \R^{2\times 2}$ is a symplectic matrix, i.e. $J^T = -J$, $J^2 = - \id$. From now on, we will use the following notation:
\[
A_f(X) \doteq Df(X)J \text{ and } B_f(X) \doteq X^TDf(X)J-f(X)J.
\]
When $f = \A$, we will simply write $A(X)$, $B(X)$ instead of $A_f(X)$, $B_f(X)$ and $C_\A$ for the set of matrices of \eqref{diffincF}. We remark that $\U$ is a solution of $\eqref{diffincF}$ with $f = \A$ if and only if $u$ solves \eqref{sys0}.
\\
\\
\indent Differential inclusions have been extensively studied in the last years, mainly in connection to the so-called convex integration methods. The underlying idea is to rewrite a system of PDE's as a relation of the form
\begin{equation}\label{inccc}
D u \in K,
\end{equation}
where $K \subset \R^{n\times 2}$ and $u \in W^{1,p}(\Omega;\R^n)$, $p \in [1,+\infty]$ and then studying solutions of the system through properties of $K$. The most natural questions we can ask about \eqref{inccc} are: which properties of $K$ guarantee that
\begin{enumerate}[(i)]
\item\label{compp} any $W^{1,p}$-equibounded sequence $u_n$ for which $\dist(D u_n,K)$ converges weakly to $0$ converges in some stronger topology and up to subsequences to a solution $u$ of $\eqref{inccc}$?
\item\label{regg} are solutions of \eqref{inccc} more regular than merely $W^{1,p}$?
\end{enumerate}
We will refer to \eqref{compp} as "compactness for approximate solutions of the differential inclusion" and to \eqref{regg} as "regularity for the differential inclusion". A necessary condition to get \eqref{compp} is that $K$ does not contain rank-one connections, i.e. $\rank(A-B) = 2, \forall A,B \in K, A \neq B$ (see \cite[Chapter 1]{KIRK}). This condition is also sufficient when $n = 2$ and $K$ is connected (as proved in \cite{SAP}), but for $n > 2$ it is not. This fact is exploited in \cite{SMVS,LSP}, where, using convex integration methods, S. M\"uller \& V. \v Sver\'ak and L. Sz{\'{e}}kelyhidi found striking counterexamples to both \eqref{compp} and \eqref{regg} in the case
\begin{equation}\label{counter}
K = 
\left\{Y\in\R^{4\times 2}: Y=
\left(
\begin{array}{c}
X\\
Df(X)J
\end{array}
\right)
\right\},
\end{equation}
$f:\R^{2\times 2} \to \R$ being a quasiconvex function (in \cite{SMVS}) or a polyconvex function (in \cite{LSP}). For the definition of quasiconvex function, we refer the reader to \cite{SMVS}. On the other hand, it is possible to prove (partial) regularity for minimizers of quasiconvex energies, see for instance \cite{EVA,KRI}. Notice that every stationary point is in particular a critical point for outer variations and every minimizer is a stationary point. Therefore, an important open question in the field of vectorial Calculus of Variations is whether for general polyconvex (or quasiconvex) functions \eqref{diffincF} admits non-regular solutions as in \cite{SMVS,LSP} or if \eqref{compp} and \eqref{regg} hold (or, without using the language of differential inclusions, if one can prove partial regularity for stationary points as in \cite{EVA,KRI}). 
\\
\\
\indent The results of this paper are in the opposite direction than the ones of \cite{SMVS,LSP}, in the sense that we prove \eqref{compp} and \eqref{regg} for particular cases of \eqref{diffincF}. Our research was driven by V. \v Sver\'ak's results in \cite{SAP}. In that article, the author studied differential inclusions
\[
D u \in K,
\]
where $K$ is a connected, compact subset of $\R^{2\times 2}$, and $u: \Omega \to \R^2$ a Lipschitz function. Roughly speaking, in order to gain \eqref{compp} (see \cite[Theorem 1]{SAP}), it is sufficient for $K$ to fulfill:
\[
c\det(X - Y) > 0,\; \forall X \neq Y, X, Y\in K,
\]
for some $c \in \R\setminus\{0\}$.
To guarantee \eqref{regg} (for instance, $D u \in C^{0,\alpha}$ for some $\alpha > 0$), a stronger inequality is needed:
\[
c\det(X - Y) \ge \|X - Y\|^2, \;\forall X,Y \in K,
\] 
for some $c \in \R\setminus\{0\}$ (see \cite[Theorem 3]{SAP}).
In view of these results, it becomes clear that in order to guarantee \eqref{compp} and \eqref{regg} one needs to carefully study the signs of the subdeterminants of matrices in $K$. Similar considerations were also made in \cite{LASLOC, FAR,LSR}.\\
\\
\indent Let us outline the structure of the paper and state our main results, in order to clarify how we use \v Sver\'ak's ideas to obtain information on our differential inclusion. In Section \ref{aarea}, we write explicitely $C_\A$ and prove basic growth estimates of $\A$. In Section $\ref{BBB}$, we show the smoothness of the solutions of the differential inclusion
\[
D w \in \{Y \in \R^{2\times 2}: Y = B(X), X \in \R^{n\times 2}\}
\]
through  Monge-Amp\`ere equation regularity result. Section \ref{BOUNDS} is devoted to the proof of Theorem $\ref{MAIN}$. This result, combined with the Monge-Amp\`ere results proved in the previous section, gives us the necessary information on the signs of the subdeterminants of the differential inclusion. Indeed in Section \ref{COMP} we prove the first of our main results:
\begin{NTeo}[Compactness of the differential inclusion]
Suppose $\U_n: \Omega \to \R^{2n + 2}$ is an equibounded sequence in $W^{1,p}(\Omega;\R^{2n+ 2})$ for $p > 2$. If
\begin{equation*}
\int_{\Omega}\dist(D\U_n(x),C_\A)\eta(x) \to 0,\; \forall \eta \in C^\infty_c(\Omega),
\end{equation*}
then, up to a (non-relabeled) subsequence, $\U_n$ converges strongly in $W^{1,\bar p}$ to a function $\U: \Omega \to \R^{2n + 2}$, for every $1\le \bar p < p$. Moreover, $D \U(x) \in C_\A$ for a.e. $x \in \Omega$.
\end{NTeo}
Finally, in Section \ref{PERT}, we prove a perturbation argument, through which we establish regularity for Lipschitz solutions of more general energies than the one induced by the area:
\begin{NTeo}
For every $R > 0$, there exists $\alpha = \alpha(R) > 0$ such that, if $f$ is a $C^k(\R^{{(2n + 2)}\times 2})$ function, $k \ge 2$, with the property that
\begin{equation*}
\|f - \A\|_{C^2(B_{2R}(0))} \le \alpha,
\end{equation*}
and $\U: \Omega \to \R^{2n+ 2}$ is a Lipschitz solution of
\begin{equation*}
D \U(x) \in C_f, \text{ for a.e. }x\in \Omega
\end{equation*}
with
\[
\|D \U\|_{\infty} \le R,
\]
then
$\U \in C^{k-1,\rho}(\Omega)$, for some positive $\rho < 1$.
\end{NTeo}

We end this paper with Section \ref{IRREG}, where we show the following:

\begin{NTeo}
Let $\Omega$ be an open and bounded subset of $\R^2$. There exists a map $\psi \in W^{1,p}(\Omega,\R^2)$ for some $p > 2$ that solves
\[
\curl(B(D\psi)) = 0,
\]
and such that for every open $\mathcal{V} \subset \Omega$, $\psi$ is not $C^1(\mathcal{V})$.
\end{NTeo}

This result shows that inner variations are in general not sufficient to guarantee any kind of regularity of solutions. The proof is based on a simple linear algebra lemma, Lemma \ref{algebra}. This is combined with the counterexample built by B. Kirchheim in \cite[Example 4.41]{KIRK}, that in turn is based on the construction of a Lipschitz non-affine map attaining exactly five matrices as gradients. There are various results in the literature concerning critical points for inner variations, and in particular showing that they are not sufficient to yield continuity of the first derivatives of solutions. In \cite{IKO}, the authors show that critical points to the inner variation equation  with positive determinant for a certain class of energies are always Lipschitz, but need not be $C^1$. In \cite{IMS,EME}, the authors construct weak singular solutions to the inner variation equations lying in the Sobolev space $W^{1,p}$, $1\le p < n$, where $n$ is the dimension of the domain. To the best of our knowledge, the aforementioned result is new for the area functional, and has been included in the paper as part of the analysis of the regularity properties of solutions to the differential inclusion associated to the area.

\subsection*{Aknowledgements}

I would like to thank my advisor, Camillo De Lellis, for posing this question and for helpful discussions. I would also like to thank Guido De Philippis, for his interest in this problem and for his suggestions, Antonio De Rosa and Maria Strazzullo, for helping me check some computations, and Yash Jhaveri, for discussions about Monge-Amp\`ere equation.

\section{The area functional}\label{aarea}

In this section we rewrite the partial differential system defining a stationary graph for the area functional as a differential inclusion. Let us consider the area functional on graphs $\mathcal{A}:\R^{n\times 2} \to \R$:
\[
\A(X) = \sqrt{1 + \|X\|^2 + \sum_{1\le a\le b\le n}\det(X^{ab})^2},
\]
where $X^{ab}$ is the $2\times 2$ submatrix obtained from $X$ considering just the $a$-th and the $b$-th rows. Denote with $\cof(M)$ the matrix for which $M\cof(M) = \det(M)\id_2$,for every $M \in \R^{2\times 2}$. Explicitely, if $$M = \left(\begin{array}{cc}a & b\\ c & d\end{array}\right),$$ then $$\cof(M) = \left(\begin{array}{cc}d & -b\\ -c & a\end{array}\right).$$ We compute
\begin{equation}\label{DD}
D\A(X) = \frac{X + \sum_{1\le a \le b \le n}\det(X^{ab})C_{ab}(X)}{\A(X)},
\end{equation}
where $C_{ab}(X)$ denotes the $n\times 2$ matrix defined as 
\begin{equation*}
(C_{ab}(X))_{ij} =
\begin{cases}
0, \text{ if } i\neq a \text{ or } i \neq b\\
(\cof(X^{ab})^T)_{ij}, \text{ otherwise.}
\end{cases}
\end{equation*}
From \eqref{DD}, it also follows that
\[
X^TD\A(X) - \A(X)\Id = \frac{X^TX + \sum_{a\le b}\det(X^{ab})^2\Id - \A^2(X)\Id}{\A(X)} =  \frac{X^TX - (1 + \|X\|^2)\Id}{\A(X)}.
\]
Let us fix the symplectic matrix to be
\[
J\doteq
\left(
\begin{array}{cc}
0 & 1\\
-1 & 0
\end{array}
\right).
\]
We will study the following particular case of $\eqref{diffincF}$:
\begin{equation}\label{areainc}
D\mathcal{U}(x) \in C_\A =
\left(
\begin{array}{cc}
X\\ 
A(X)\\
B(X)
\end{array}
\right), \text{ for a.e. }x\in \Omega
\end{equation}
where $\mathcal{U}: \Omega \to \R^{2n + 2}$ is a function in a Sobolev space (its regularity will be discussed at the end of this section). We will always use the following notation for a map $\U$ with the property \eqref{areainc}:
\[
\U =
\left(
\begin{array}{c}
u\\
v\\
w
\end{array}
\right),\;
u,v:\Omega\to \R^n,w:\Omega\to\R^2,
\]
so that $\U$ satisfies \eqref{areainc} if and only if $D v(x) = A(D u(x))$ and $D w(x) = B(D u(x))$ for a.e. $x \in \Omega$. Let us make some preliminary computations that we will need in the paper. Namely:
\begin{lemma}\label{basicest}
The following hold
\begin{enumerate}
\item $\|A(X)\| \le 2\|X\|$;\label{AEST}
\item $\frac{1 + \|X\|^2}{2\A(X)}\le\|B(X)\|\le 2(1 +\|X\|)$;\label{BEST}
\end{enumerate}
\end{lemma}
\begin{proof}
In this proof, we will make use of the Cauchy-Binet Theorem (see \cite[Proposition 2.69]{AFP}), that asserts the identity $\sum_{1\le a \le b \le n}\det(X^{ab})^2 = \det(X^TX)$. To prove $\eqref{AEST}$, we write
\begin{align*}
\|A(X)\|^2 &= \frac{\|X + \sum_{1\le a \le b \le n}\det(X^{ab})C_{ab}^T(X)\|^2}{(1 + \|X\|^2 + \det(X^TX)}\\
& \le 2 \frac{\|X\|^2 + \sum_{1\le a \le b \le n}\det(X^{ab})^2\|C_{ab}^T(X)\|^2}{1 + \|X\|^2 + \det(X^TX)}\\
&\le 2 \frac{\|X\|^2 + \sum_{1\le a \le b \le n}\det(X^{ab})^2\|X^{ab}\|^2}{1 + \|X\|^2 + \det(X^TX)}\\
&\le 2\|X\|^2 \frac{1 + \sum_{1\le a \le b \le n}\det(X^{ab})^2}{1 + \|X\|^2 + \det(X^TX)}\\
&= 2\|X\|^2 \frac{1 + \det(X^TX)}{1 + \|X\|^2 + \det(X^TX)} < 2\|X\|^2.
\end{align*}
To prove $\eqref{BEST}$, we again write
\begin{align*}
\|B(X)\|^2 &= \frac{\|X^TX - (1 + \|X\|^2)\id_2\|^2}{1 + \|X\|^2 + \det(X^TX)}\\
& = \frac{\|X^TX\|^2 + 2(1 + \|X\|^2)^2 -2\|X\|^2 - 2\|X\|^4}{1 + \|X\|^2 + \det(X^TX)}\\
& = \frac{\|X^TX\|^2 + 2 + 2\|X\|^2}{1 + \|X\|^2 + \det(X^TX)}.
\end{align*}
It is easy to see that
\[
\frac{1}{4}\|X\|^4 \le \|X^TX\|^2 \le 4\|X\|^4.
\]
Therefore, we get the estimates
\[
\frac{4^{-1}\|X\|^4 + 2 + 2\|X\|^2}{1 + \|X\|^2 + \det(X^TX)}\le\|B(X)\|^2 \le \frac{4\|X\|^4 + 2 + 2\|X\|^2}{1 + \|X\|^2 + \det(X^TX)},
\]
and we deduce that
\[
\frac{1}{4}\frac{\|X\|^4 + 1 + 2\|X\|^2}{1 + \|X\|^2 + \det(X^TX)}\le\|B(X)\|^2 \le 4\frac{\|X\|^4 + 1 + 2\|X\|^2}{1 + \|X\|^2 + \det(X^TX)}.
\]
Using the fact that $\det(X^TX) \ge 0$, rewriting $\|X\|^4 + 1 + 2\|X\|^2 = (1 + \|X\|^2)^2$, and taking the square root of the terms of the inequalities, we get
\[
\frac{1 + \|X\|^2}{2\A(X)}\le\|B(X)\| \le \sqrt{4}\frac{ 1 + \|X\|^2}{\sqrt{1 + \|X\|^2}} = 2\sqrt{1 + \|X\|^2}.
\]
Hence also the second estimate is proven.
\end{proof}
With the previous lemma, we immediately get
\begin{Cor}
For any $p\ge 1$, if $u \in W^{1,p}(\Omega)$, and $\U$ satisfies \eqref{areainc}, then $\U \in W^{1,p}(\Omega)$.
\end{Cor}

\section{Properties of $B(\cdot)$}\label{BBB}

In this section, we prove some properties of the matrix field $B(X)$. In Proposition \ref{firstreg} we show how these imply the smoothness of the function $w$ in $\eqref{areainc}$. We recall that
\begin{equation}\label{Bform}
B(X) = \frac{X^TXJ - (1 + \|X\|^2)J}{\A(X)}.
\end{equation}
\begin{lemma}\label{Bprop}
The following properties hold:
\begin{enumerate}[(i)]
\item $\tr(B(X)) = 0, \forall X$;\label{tr}
\item $B(X)_{12} <0, B(X)_{21} > 0, \forall X$;\label{signs}
\item $\det(B(X)) = 1$.\label{1}
\end{enumerate}
\end{lemma}
\begin{proof}
Let us write $B(X)$ explicitely. Denote with $X^1$, $X^2$ the column vectors of $\R^n$ representing the columns of the matrix $X$. First,
\[
X^TXJ =
\left(
\begin{array}{cc}
\|X^1\|^2 & (X^1,X^2)\\
(X^1,X^2) & \|X^2\|^2
\end{array}
\right)J
=
\left(
\begin{array}{cc}
-(X^1,X^2)& \|X^1\|^2\\
-\|X^2\|^2& (X^1,X^2)
\end{array}
\right).
\]
Therefore,
\[
\A(X)B(X) = X^TXJ - (1 + \|X\|^2)J=
\left(
\begin{array}{cc}
-(X^1,X^2)& \|X^1\|^2\\
-\|X^2\|^2& (X^1,X^2)
\end{array}
\right)
-
\left(
\begin{array}{cc}
0& 1+\|X\|^2\\
-1-\|X\|^2& 0
\end{array}
\right),
\]
and
\begin{equation}\label{B}
\A(X)B(X) = X^TXJ - (1 + \|X\|^2)J=
\left(
\begin{array}{cc}
-(X^1,X^2)& -1 - \|X^2\|^2\\
1 +\|X^1\|^2& (X^1,X^2)
\end{array}
\right).
\end{equation}
Since $\mathcal{A}(X)$ is always positive, we can divide the previous expressions by $\mathcal{A}(X)$ to infer \eqref{tr} and \eqref{signs}. In order to prove the third property, we compute:
\[
\A^2(X)\det(B(X)) = 1 + \|X\|^2 + \|X^1\|^2\|X^2\|^2 - (X^1,X^2)^2 = 1 + \|X\|^2 +\det(X^TX) = \A^2(X).
\]
Again the positivity of $\mathcal{A}(X)$ implies the conclusion of \eqref{1}.
\end{proof}

We now consider properties of the differential inclusion
\begin{equation}\label{incl}
D w(x) = B(D u(x)), \text{ for a.e. } x \in \Omega
\end{equation}
for $w \in W^{1,2}(\Omega)$. By $\eqref{tr}$ of Lemma \ref{Bprop} we have $\dv(w) = 0$. Therefore, $w = (w_1,w_2)$ can be rewritten as
\[
w = (-\partial_2z,\partial_1z)
\]
for some $z \in W^{2,2}(\Omega)$. Consequently, \eqref{incl} is rewritten as
\[
\left(
\begin{array}{cc}
-\partial_{12}z&-\partial_{22}z\\
\partial_{11}z& \partial_{12}z
\end{array}
\right)
= B(D u).
\]
Using properties $\eqref{signs}$ and $\eqref{1}$ of Lemma \ref{Bprop}, we find that $z$ enjoys the following properties
\begin{equation}\label{Monge}
\begin{cases}
\det(D^2z) = 1 \text{ a.e.},\\
\Delta z > 0 \text{ a.e.},\\
z \in W^{2,2}(\Omega).
\end{cases}
\end{equation}

In the next Proposition, we will exploit some fundamental results concerning solutions to the Monge-Amp\`ere equation. We refer the reader to \cite{FIG} for the definitions and the results we will use. In particular, we refer the reader to \cite[Definition 2.1]{FIG} for the definition of Monge-Amp\`ere measure. Here and in the rest of the paper, we denote with $\mathcal{L}^m$ the $m$-dimensional Lebesgue measure in $\R^m$.

\begin{prop}\label{firstreg}
Suppose $z$ solves $\eqref{Monge}$. Then, $z$ is smooth.
\end{prop}
\begin{proof}
We just need to prove that $z$ is an Alexandrov solution of the Monge-Amp\`ere equation, and then apply the classical regularity results for the Monge-Amp\`ere equation. It is not restrictive to prove the result on balls $B_r(\bar x) \subset \Omega$ such that $B_r(\bar x) \subset B_R(\bar x) \subset \Omega$. Consider a standard mollification kernel $\rho_\varepsilon$, i.e. $\rho_\varepsilon \in C^\infty_c(\R^2)$, $\spt(\rho_\varepsilon) \subset B_\varepsilon(0)$, $\rho_\varepsilon \ge 0$, $\int_{\R^2}\rho_\varepsilon(x)\dx = 1$ for every $\varepsilon > 0$. Finally, define $z_\varepsilon(x) \doteq (z\star \rho_\varepsilon)(x)$, for $\varepsilon \le \frac{R - r}{2}$. We exploit the embedding
\begin{equation}\label{emb}
C^0(\Omega)\cap L^\infty_{\loc}(\Omega) \subset W^{2,2}(\Omega)
\end{equation}
to argue that $z$ is continuous in $B_r(\bar x)$. We also prove that it is convex on $B_r(\bar x)$. For every $x \in B_r(\bar x)$ and for every $v \in \R^2$, we compute
\[
(D^2z_\varepsilon(x)v,v) = \int_{\R^{2}}\rho_\varepsilon(y + x) (D^2z(y)v,v)\dy > 0.
\]
Therefore, $z_\varepsilon$ is a sequence of convex functions converging in the $C^0(B_r(\bar x))$ topology to $z$. Thus, $z$ must be convex too. Denote with $\mu_z$ and $\mu_{z_\varepsilon}$ the Monge-Amp\`ere measures associated to $z$ and $z_\varepsilon$ respectively. We need to show that
\[
\mu_z = \det(D^2 z)\mathcal{L}^2\llcorner B_r(\bar x).
\]
To do so, first we notice that the $W^{2,2}$ convergence of $z_\varepsilon$ to $z$ imply that $\det(D^2z_\varepsilon) \to \det(D^2z)$ in the $L^1$- norm. Moreover we use \cite[Proposition 2.6]{FIG} to infer that the Monge-Amp\`ere measures associated to $z_\varepsilon$ converge weakly in the sense of measures to the Monge-Amp\`ere measure associated to $z$. From the regularity of $z_\varepsilon$, we infer $\mu_{z_\varepsilon} = \det(D^2 z_\varepsilon)\mathcal{L}^2\llcorner B_r(\bar x)$, hence for every $g \in C_c(B_r(\bar x))$ we have:
\[
\int_{B_r(\bar x)}gd\mu_\varepsilon = \int_{B_r(\bar x)}g(x)\det(D^2z_{\varepsilon})(x)\dx \to\int_{B_r(\bar x)}g\det(D^2 z)\dx
\]
and
\[
\int_{B_r(\bar x)}gd\mu_\varepsilon \to \int_{B_r(\bar x)} gd\mu.
\]
We infer $\mu = \det(D^2 z)\mathcal{L}^2\llcorner{B_r(\bar x)} =\mathcal{L}^2\llcorner{B_r(\bar x)}$. Hence, $z$ is an Alexandrov solution to $\det(D^2z) = 1$. It follows that $z$ is strictly convex by  \cite[Theorem 2.19]{FIG} and smooth by \cite[Theorem 3.10]{FIG}.
\end{proof}

Let us conclude this section with another important property of $B(X)$ that follows from a direct computation, see also \cite[Section 6]{SAP}:

\begin{prop}\label{alg}
For all $R > 0$, there exists $\mu = \mu(R)> 0$ such that if $\|X\|,\|Y\|\le R$, then
\begin{equation}\label{algg}
\det(B(X)-B(Y)) \le -\mu\|B(X) - B(Y)\|^2.
\end{equation}
\end{prop}
\section{Bounds on the subdeterminants and regularity}\label{BOUNDS}

\begin{Teo}\label{MAIN}
For every number $k \ge 0$ there exists positive numbers $C(k),\delta(k) > 0$ such that for every couple $(X,Y) \in \R^{n\times 2}\times \R^{n\times 2}$ the following holds:
\begin{equation}\label{sotto}
-\langle (A(X) - A(Y))J,X - Y \rangle + C\|B(X) - B(Y)\|\min\{\|Y\|,\|X\|\}\|X - Y\| \ge \delta\|X - Y\|^2,
\end{equation}
provided that
\begin{align*}
\max\{\|B(X)\|,\|B(Y)\|\} \le k.
\end{align*}
\end{Teo}

\begin{rem}\label{not}
Let us use the following notation: $\alpha(X) \doteq -B_{12}(X)$, $\beta(X) \doteq B_{21}(X)$, $\gamma(X) \doteq -B_{11}(X)$. These functions were explicitly written in Lemma \ref{Bprop}. Notice that, as it was proved in $\eqref{1}$ of Lemma \ref{Bprop}:
\begin{equation}\label{rel}
\alpha(X)\beta(X) - \gamma^2(X) = \det(B(X)) = 1,\forall X\in \R^{n\times 2}
\end{equation}
\end{rem}
\begin{proof}
For a matrix $M \in \R^{n\times 2}$, we use the notation $$M = \left(\begin{array}{cc}m_{11} &m_{12}\\ m_{21} & m_{22} \\ \dots & \dots \\ m_{n1} & m_{n2}\end{array}\right),$$
and we write $M^1$, $M^2$ for the first and second column of $M$, respectively, i.e.
\[
M^1 = \left(\begin{array}{c}m_{11} \\m_{21}\\ \dots \\ m_{n1}\end{array}\right) \quad \text{ and }\quad M^2 = \left(\begin{array}{c}m_{12}\\ m_{22}\\ \dots \\ m_{n2}\end{array}\right)
\]
First of all, we compute
\begin{align*}
\A(X)D\A(X)_{j1} &= x_{j1} -\sum_{i = 1}^{j - 1}x_{i2}(x_{i1}x_{j2} - x_{i2}x_{j1}) + \sum_{i = j}^{n}x_{i2}(x_{j1}x_{i2} - x_{j2}x_{i1})\\
 &= x_{j1} -\sum_{i = 1}^{j - 1}x_{i2}(x_{i1}x_{j2} - x_{i2}x_{j1}) + \sum_{i = j}^{n}x_{i2}(x_{j1}x_{i2} - x_{j2}x_{i1})\\
& = x_{j1}(1 + \|X^2\|^2) -  (X^1,X^2)x_{j2}
\end{align*}

and

\begin{align*}
\A(X)D\A(X)_{j2} &= x_{j2} +\sum_{i = 1}^{j - 1}x_{i1}(x_{i1}x_{j2} - x_{i2}x_{j1}) - \sum_{i = j}^{n}x_{i1}(x_{j1}x_{i2} - x_{j2}x_{i1})\\
 &= x_{j2}(1 + \|X^1\|^2) -(X^1,X^2)x_{j1}.
\end{align*}

Using the notation of Remark \ref{not}
\[
D\A(X)_{j1} = \beta(X)x_{j1} - \gamma(X)x_{j2} \text{ and } D\A(X)_{j2} = \alpha(X)x_{j2} - \gamma(X)x_{j1}.
\]

Assume, without loss of generality, that $\|X\| \ge \|Y\|$. We can write

\begin{align*}
&(D\A(X)_{j1} - D\A(Y)_{j1})(x_{j1} - y_{j1})\\
&=(\beta(X)x_{j1} - \beta(Y)y_{j1})(x_{j1} - y_{j1}) - (\gamma(X)x_{j2} - \gamma(Y)y_{j2})(x_{j1} - y_{j1})\\
&=\beta(X)(x_{j1} - y_{j1})^2 + (\beta(X) - \beta(Y))y_{j1}(x_{j1}-y_{j1})  - (\gamma(X)x_{j2} - \gamma(Y)y_{j2})(x_{j1} - y_{j1})\\
&=\beta(X)(x_{j1} - y_{j1})^2 + (\beta(X) - \beta(Y))y_{j1}(x_{j1}-y_{j1}) - \gamma(X)(x_{j2} - y_{j2})(x_{j1} - y_{j1}) \\
&+ (\gamma(Y) - \gamma(X))y_{j2}(x_{j1} - y_{j1})
\end{align*}
and
\begin{align*}
&(D\A(X)_{j2} - D\A(Y)_{j2})(x_{j2} - y_{j2}) \\
&= (\alpha(X)x_{j2} - \alpha(Y)y_{j2})(x_{j2} - y_{j2}) - (\gamma(X)x_{j1} - \gamma(Y)y_{j1})(x_{j2} - y_{j2})\\
&=\alpha(X)(x_{j2} - y_{j2})^2 + (\alpha(X) - \alpha(Y))y_{j2}(x_{j2}-y_{j2})  - (\gamma(X)x_{j1} - \gamma(Y)y_{j1})(x_{j2} - y_{j2})\\
&=\alpha(X)(x_{j2} - y_{j2})^2 + (\alpha(X) - \alpha(Y))y_{j2}(x_{j2}-y_{j2})  - \gamma(X)(x_{j1} - y_{j1})(x_{j2} - y_{j2})\\
&+(\gamma(Y) - \gamma(X))y_{j1}(x_{j2} - y_{j2}).
\end{align*}
Therefore
\begin{equation}\label{bigsum}
\begin{split}
&-\langle (A(X) - A(Y))J,X - Y \rangle = \langle D\A(X) - D\A(Y),X - Y \rangle \\
&=\sum_{j = 1}^n(D\A(X)_{j1} - D\A(Y)_{j1})(x_{j1} - y_{j1}) + \sum_{j = 1}^n(D\A(X)_{j2} - D\A(Y)_{j2})(x_{j2} - y_{j2}) \\
& =\sum_j\beta(X)(x_{j1} - y_{j1})^2 - 2\gamma(X)(x_{j2} - y_{j2})(x_{j1} - y_{j1}) + \alpha(X)(x_{j2} - y_{j2})^2 \\
& +(\gamma(Y) - \gamma(X))y_{j2}(x_{j1} - y_{j1})+ (\alpha(X) - \alpha(Y))y_{j2}(x_{j2}-y_{j2}) \\
&+(\gamma(Y) - \gamma(X))y_{j1}(x_{j2} - y_{j2})+ (\beta(X) - \beta(Y))y_{j1}(x_{j1}-y_{j1}).
\end{split}
\end{equation}

 First, we claim that there exists a constant $\delta = \delta(k)$ independent of $X$ such that, for every $X$ for which $\|B(X)\|\le k$ and for every $a,b \in \R$
\begin{equation}\label{poll}
-2|\gamma(X)| ab + \beta(X) a^2 + \alpha(X) b^2 \ge \delta(a^2 + b^2).
\end{equation}
Fix $X$. Since $\alpha(X) + \beta(X) \ge 2$, either $\beta(X) \ge 1$ or $\alpha(X) \ge 1$. Without loss of generality, we can suppose $\beta(X) \ge 1$. Therefore, if $b = 0$, we can choose any $\delta < 1$. If $b \neq 0$, we divide the expression by $b^2$ and claim \eqref{poll} becomes equivalent to
\[
(\beta(X) - \delta) x^2 -2|\gamma(X)| x + (\alpha(X) - \delta) \ge 0, \forall x \in \R.
\]
Taking into account \eqref{rel}, i.e. $\gamma^2 = \alpha\beta - 1$, the discriminant of the previous equation becomes
\[
\Delta(X)_\delta = 4\gamma^2 - 4(\alpha(X)-\delta)(\beta(X)-\delta) = -4 -4\delta^2 +4\delta(\alpha(X) + \beta(X)).
\]
Since $\beta(X)$ and $\alpha(X)$ are both uniformly bounded, we can choose some small $\delta < 1$ depending only on $k$ (so, in particular, independent of $X$) for which $\Delta(X)_\delta< 0$ for every $X$ such that $\|B(X)\|\le k$. This implies that the polynomial $x\mapsto(\beta(X) - \delta) x^2 -2\gamma(X) x + (\alpha(X) - \delta)$ has no real root. Since $\beta(X) \ge 1 > \delta$ by assumption, then the polynomial is positive for large values of $x$, therefore it is positive everywhere, as we wanted. Having shown the claim, we can apply inequality \eqref{poll} with $a = \sqrt{\sum_{j}(x_{j1} - y_{j1})^2}$ and $b = \sqrt{\sum_{j}(x_{j2} - y_{j2})^2}$ to deduce that
\begin{equation}\label{fterm}
\begin{split}
&\sum_j(\beta(X)(x_{j1} - y_{j1})^2 - 2\gamma(X)(x_{j2} - y_{j2})(x_{j1} - y_{j1}) + \alpha(X)(x_{j2} - y_{j2})^2)  \\
&\ge\beta(X)\sum_j(x_{j1} - y_{j1})^2 + \alpha(X)\sum_{j}(x_{j2} - y_{j2})^2 -2|\gamma(X)|\sqrt{\sum_{j}(x_{j1} - y_{j1})^2}\sqrt{\sum_{j}(x_{j2} - y_{j2})^2})\\
& \ge\delta\sum_j((x_{j1} - y_{j1})^2 + (x_{j2} - y_{j2})^2) = \delta\|X - Y\|^2.
\end{split}
\end{equation}

We also estimate:
\begin{equation}\label{sterm}
\begin{split}
(\gamma(Y) - \gamma(X))y_{j2}(x_{j1} - y_{j1}) \ge -|\gamma(Y) - \gamma(X)|\|Y\|\|X - Y\|,\\
(\alpha(X) - \alpha(Y))y_{j2}(x_{j2}-y_{j2})  \ge -|\alpha(Y) - \alpha(X)|\|Y\|\|X - Y\|,\\
(\gamma(Y) - \gamma(X))y_{j1}(x_{j2} - y_{j2})\ge -|\gamma(Y) - \gamma(X)|\|Y\|\|X - Y\|,\\
(\beta(X) - \beta(Y))y_{j1}(x_{j1}-y_{j1})\ge -|\beta(Y) - \beta(X)|\|Y\|\|X - Y\|.
\end{split}
\end{equation}

By the definition of $\alpha,\beta$ and $\gamma$, $2|\gamma(Y) - \gamma(X)| + |\alpha(Y) - \alpha(X)| + |\beta(Y) - \beta(X)| \le C_1\|B(X) - B(Y)\|$, where $C_1 > 0$ is an universal constant. Combining \eqref{fterm} and \eqref{sterm}, we finally estimate in \eqref{bigsum}:

\begin{align*}
& \sum_j\left(\beta(X)(x_{j1} - y_{j1})^2 - 2\gamma(X)(x_{j2} - y_{j2})(x_{j1} - y_{j1}) + \alpha(X)(x_{j2} - y_{j2})^2\right) \\
& + \sum_j(\gamma(Y) - \gamma(X))y_{j2}(x_{j1} - y_{j1})+ \sum_j(\alpha(X) - \alpha(Y))y_{j2}(x_{j2}-y_{j2})\\
& +\sum_j(\gamma(Y) - \gamma(X))y_{j1}(x_{j2} - y_{j2})+ \sum_j(\beta(X) - \beta(Y))y_{j1}(x_{j1}-y_{j1})\\
& \ge \delta\|X - Y\|^2 - nC_1\|B(X) - B(Y)\|\|Y\|\|X - Y\|.
\end{align*}
This estimate completes the proof of \eqref{sotto}.
\end{proof}

\subsection{Regularity of the Differential Inclusion}

The regularity of $W^{1,2}$ solutions of \eqref{areainc} is surely a well-known result to the experts of the field. Since we could not find a reference of this fact in the literature and the argument is very short, we give a proof here.

\begin{prop}\label{secondreg}
Every $W^{1,2}$ solution $\U$ of \eqref{areainc} is smooth.
\end{prop}

\begin{proof}
From the proof of the previous theorem, we know that
\[
D\A(X)_{j1} = \beta(X)x_{j1} - \gamma(X)x_{j2} \text{ and } D\A(X)_{j2} = \alpha(X)x_{j2} - \gamma(X)x_{j1}.
\]
The equation
\[
\dv(D\A(D u)) = 0
\]
reads, for every $j \in \{1,\dots,n\}$,
\begin{equation}\label{eq}
\partial_{1}(\beta(D u)\partial_1u^j - \gamma(D u)\partial_2u^j) + \partial_2(\alpha(D u)\partial_2u^j - \gamma(D u)\partial_1u^j) = 0,
\end{equation}
where $u = (u^1,\dots,u^n)$. The previous equation has to be intended in the weak sense. In \eqref{firstreg} it is showed that $\alpha(D u)$, $\beta(D u)$, $\gamma(D u)$ are smooth functions. Moreover, the matrix
\[
M(D u) = (BJ)^T(D u) = 
\left(
\begin{array}{cc}
\beta(D u) & -\gamma(D u)\\
-\gamma(D u) & \alpha(D u)
\end{array}
\right)
\]
is locally bounded in the sense of quadratic forms above and below by
\begin{equation}\label{boundds}
c_1\id \le M(D u(x)) \le c_2\id
\end{equation}
for two positive constants $c_1\le c_2$. The argument to prove \eqref{boundds} is exactly the same as the one used to prove \eqref{poll}. Therefore, every $u^j$ is the weak solution to a second order elliptic equation with smooth coefficients, \eqref{eq}. It is well known that solutions to this class of equations are smooth. 
\end{proof}

\begin{rem}
This is not the first time that regularity results for the Monge-Amp\`ere equation have been exploited to obtain regularity for the minimal surface equation. In \cite{NI}, this connection is used to prove Bernstein's theorem (i.e., that the only solution to the minimal surface equation/system in the whole $\R^2$ are affine functions) for $2$-dimensional minimal graphs in $\R^3$. We remark that, in view of the well-known Bernstein property for solutions of Monge-Amp\`ere equation (see \cite{NI}), Proposition \ref{firstreg} and Proposition \ref{secondreg} immediately give Bernstein's property for $W^{1,\infty}$ $2$-dimensional minimal graphs in $\R^{n + 2}$.
\end{rem}

\section{Compactness of the differential inclusion in $W^{1,p}$, $p > 2$}\label{COMP}

The main result of this section is Theorem \ref{Cor}, where we prove the compactness of the differential inclusion \eqref{areainc}. First, we recall some results about Young measure.

\subsection{Preliminaries: Young measures}

The results we report here are taken from \cite[Section 3]{DMU}, to which we refer the interested reader for a more detailed exposition of the subject. We will denote with $\mathcal{M}(\R^m)$ the space of finite and positive measures on $\R^m$.

\begin{Teo}[Fundamental Theorem on Young measure]
Let $E\subset \R^n$ be a Lebesgue measurable set with finite measure. Consider a sequence $z_j:E\subset \R^d \to \R^N$ of measurable functions satisfying the condition
\[
\sup_{j \in \N}\int_E\|z_j\|^s< +\infty,
\]
for some $s>0$. Then there exists a subsequence $z_{j_k}$ and a weak-* measurable map $\nu:E \to \mathcal{M}(\R^N)$ such that for $\mathcal{L}^d$-a.e. $x \in E$, $\nu_x \in \mathcal{M}(\R^N)$ and in addition $\nu_x(\R^N) = 1$. Moreover, for every $A \subset E$, and for every $f \in C(\R^N)$, if
\[
f(z_{j_k}) \text{ is relatively weakly compact in } L^1(A),
\]
then,
\[
f(z_{j_k}) \rightharpoonup \bar f \text{ in } L^1(A), \text{ where } \bar f(x) =\langle \nu_x, f\rangle = \int_{\R^N}f(y)d\nu_x(y).
\]
In this case, we say that $z_{j_k}$ {generates the Young measure} $\nu$.
\end{Teo}

\begin{Cor}\label{Youngp}
Let $p > 1$ and $E\subset \R^d$ be a Lebesgue measurable set with finite measure. If $z_j$ is weakly convergent in $L^p(E)$ to a function $z \in L^p(E)$ and if it generates the Young measure $\nu$, then, for every $f \in C(\R^N)$ such that
\[
|f(y)|\le C(1 + \|y\|^q), \text{ for } q < p,
\]
the following holds
\[
f(z_j) \rightharpoonup \bar{f}, \text{weakly in } L^{\frac{p}{q}}(E).
\]
In particular, the choice $f$ such that $f(y) = y,\; \forall y \in \R^N$ yields
\begin{equation}\label{exp}
z(x) = \langle\nu_x,f\rangle.
\end{equation}
\end{Cor}

Another result, fundamental to establish compactness, is the following \cite[Corollary 3.2]{DMU}:

\begin{Cor}\label{strongc}
Suppose that a sequence $z_j$ of measurable functions from $E$ to $\R^N$ generates the Young measure $\nu$. Then 
\[
z_j \to z \text{ in measure if and only if } \nu_x = \delta_{z(x)} \text{ for $\mathcal{L}^d$-a.e }x.
\]
In particular, if $z_j \in L^p(E)$, for $p >1$, and  the following hold
\begin{enumerate}[(i)]
\item $z_j$ is weakly convergent in $L^p(E)$ to a function $z \in L^p(E)$,
\item $z_j$ generates the Young measure $\nu$,
\item $\nu_x = \delta_{z(x)} \text{ for $\mathcal{L}^d$-a.e }x.$,
\end{enumerate}
Then,
\[
z_j \to z \text{ in $L^q(E)$}, \text{ for every $1\le q <p$}.
\]
\end{Cor}

\subsection{Compactness results}

We will make use of the following identity, that can be easily checked by direct computation

\begin{equation}\label{linearalg}
\langle X,YJ\rangle = -\sum_{i = 1}^m\det
\left(\begin{array}{c}
X_i\\
Y_i
\end{array}\right)
\end{equation}
 for every $X,Y \in \R^{m\times 2}$, where $X_i$, $Y_i$ are the $i$-th rows of the matrices $X$ and $Y$.

\begin{Teo}[Compactness of the differential inclusion]\label{Cor}
Suppose $\U_n: \Omega \to \R^{2n + 2}$ is an equibounded sequence in $W^{1,p}(\Omega;\R^{2n+ 2})$ for $p > 2$. If
\begin{equation}\label{meas1}
\int_{\Omega}\dist(D\U_n(x),C_\A)\eta(x) \to 0,\; \forall \eta \in C^\infty_c(\Omega),
\end{equation}
then, up to a (non-relabeled) subsequence, $\U_n$ converges strongly in $W^{1,\bar p}$ to a function $\U: \Omega \to \R^{2n + 2}$, for every $1\le \bar p < p$. Moreover, $D \U(x) \in C_\A$ for a.e. $x \in \Omega$.
\end{Teo}
\begin{proof}
Throughout the proof, we will use the splitting
\begin{equation}\label{split}
\Lambda=
\left(
\begin{array}{c}
\Lambda_1\\
\Lambda_2\\
\Lambda_3
\end{array}
\right),\; \Lambda_1, \Lambda_2  \in \R^{n\times 2}, \Lambda_3 \in \R^{2\times 2}
\end{equation}
for every $\Lambda \in \R^{(2n + 2)\times 2}$. We can assume that $\U_n$ converges weakly in $W^{1,p}$ to $\U$, and that $D\U_n$ converges in the sense of Young measures to $\{\nu_x\}_x$. We claim that, for almost every $x\in \Omega$, we have
\begin{enumerate}[(i)]
\item  $\spt(\nu_x) \subseteq C_\A$;\label{supp}
\item $\int_{\R^{2n + 2}}\det(\Lambda^{ab})d\nu_x(\Lambda) = \det\left(\left(\int_{\R^{2n + 2}}\Lambda d\nu_x\right)^{ab}\right)$, $\forall 1\le a\le b \le 2n + 2 $.\label{DET}
\end{enumerate}

To prove the previous claim, it just suffices to apply the definition of Young measure generated by $\U_n$. Indeed to show \eqref{supp} consider the function $f \in C(\R^{(2n + 2)\times 2})$ defined as $f(\Lambda) \doteq \dist(\Lambda,C_\A)$. The proof of \eqref{DET} is analogous to the one given in \cite[Theorem 1]{SAP}.
Moreover, using the equality
\begin{equation}\label{sumdet}
\det(M_1+ M_2) = \det(M_1) + \det(M_2) + \langle M_1,\cof^T(M_2)\rangle,
\end{equation}
valid for every matrices $M_1,M_2 \in \R^{2\times 2}$, and $\eqref{DET}$ of the previous claim, it is easy to see that
\[
\int_{\R^{(2n + 2)\times 2}\times \R^{(2n + 2)\times 2}}\det((\Lambda - \Gamma)^{ab})d(\nu_x(\Lambda)\otimes \nu_x(\Gamma)) = 0\; \text{ for a.e. }x\in \Omega,
\]
where $\nu_x\otimes \nu_x$ denotes the standard product measure constructed with $\nu_x$. Clearly this implies that for any collection of numbers $t_{ab} \in \R$,
\begin{equation}\label{lincomb}
\sum_{1\le a\le b \le 2n + 2}t_{ab}\int_{\R^{(2n + 2)\times 2}\times \R^{(2n + 2)\times 2}}\det((\Lambda -\Gamma)^{ab})d(\nu_x(\Lambda)\otimes \nu_x(\Gamma)) = 0.
\end{equation}
First, we choose $t_{ab} = 0$ for every $1\le a \le b \le 2n$ and $t_{ab} = 1$ if $a = 2n + 1, b = 2n + 2$. Using $\eqref{supp}$ of the claim and \eqref{algg}, we infer that $\nu_x\otimes \nu_x$ is supported in the set of matrices
\[
C_\A\times C_\A \cap \{(\Lambda',\Lambda'') \in \R^{(2n + 2)\times 2}\times \R^{(2n + 2)\times 2}: \Lambda'_3 = \Lambda''_3\}.
\]
Thus, we obtain the existence of a $2\times 2$ matrix $B_x$ such that  $B(\Lambda_1) = B_x$ for a.e. $x \in \Omega$ and for $\nu_x$-a.e. $\Lambda\in \R^{2n + 2}$. Let us remark that the matrix $B_x$ possibly depends on $x \in \Omega$ but not on $\Lambda \in \R^{(2n + 2)\times 2}$. To finish the proof, apply \eqref{linearalg} to find coefficients $t_{ab}$ such that
\[
\sum_{1\le a\le b \le 2n + 2}t_{ab}\det((\Lambda -\Gamma)^{ab}) = \langle (A(\Lambda_1) - A(\Gamma_1))J,\Lambda_1 - \Gamma_1\rangle, \forall \Lambda,\Gamma \in \R^{(2n + 2)\times 2}.
\]
Now we can use \eqref{sotto} to infer that for a.e. $x \in \Omega$, there exists a number $\delta(x) > 0$
\begin{align*}
0 &= \int_{\R^{(2n + 2)\times 2}\times \R^{(2n + 2)\times 2}} \langle (A(\Lambda_1) - A(\Gamma_1))J,\Lambda - \Gamma\rangle d(\nu_x(\Lambda)\otimes \nu_x(\Gamma)) \\
&\ge \int_{\R^{(2n + 2)\times 2}\times \R^{(2n + 2)\times 2}} \delta(x)\|\Lambda_1 - \Gamma_1\|^2 d(\nu_x(\Lambda)\otimes \nu_x(\Gamma)).
\end{align*}

This yields $\nu_x = \delta_{D\U(x)}$ for a.e. $x \in \Omega$. Corollary \ref{strongc} implies that $D \U_n$ converges in measure to $D\U$ and therefore strongly for every $1\le \bar p < p$.

\end{proof}

\section{Perturbative result}\label{PERT}

We will prove that solutions with fixed Lipschitz constant of the differential inclusion $\eqref{diffincF}$ for functionals sufficiently near to the area functional are actually as smooth as the functional under consideration. The strategy is the following. In Lemma \ref{pertar}, we prove inequality \eqref{R}, through which we bound the norm of the difference of two matrices with a linear combination of subdeterminants of $C_{\A}$. Next, in Lemma \ref{pertlem}, we show that, if we fix $R > 0$, there exists a number $\varepsilon(R) > 0$ such that, if $f: \R^{n\times 2}\to\R$ is a $C^2$ functional with $\|f - \A\|_{C^2(B_{2R})} \le \varepsilon(R)$, then for $f$ the same kind of inequality holds (see \eqref{genf}). In Theorem \ref{pert} and Proposition \ref{C1gamma}, we show how inequality \eqref{genf} implies H\"older continuity of gradients of functions $\mathcal{U}$ satisfying
\[
D \U(x) \in C_f, \text{ for a.e. }x\in \Omega. 
\]
Finally, in Subsection \ref{mor}, we will improve the H\"older continuity of the gradient of the solution to higher regularity.

\begin{lemma}\label{pertar}
For every $R >0$, there exist constants $\lambda(R),\delta(R) >0$ such that, $\forall X,Y \in B_{\frac{3R}{2}}(0)$, we have
\begin{equation}\label{R}
-\langle (A(X)-A(Y))J,X - Y\rangle - \lambda\det(B(X) - B(Y)) \ge \delta\|X - Y\|^2,
\end{equation}
\end{lemma}

\begin{proof}
We note that for $(X,Y) \in B_\frac{3R}{2}(0)\times B_\frac{3R}{2}(0)$ the assumptions of Theorem \ref{MAIN} are fulfilled. Therefore, we find constants $C = C(R)$ and $c = c(R)$ such that
\[
-\langle (A(X) - A(Y))J,X - Y \rangle + C\|B(X) - B(Y)\|\min\{\|Y\|,\|X\|\}\|X - Y\| \ge c\|X - Y\|^2.
\]
Using the hypothesis, we estimate $\min\{\|Y\|,\|X\|\} \le \max\{\|Y\|,\|X\|\} \le \frac{3R}{2}$. Moreover Young inequality yields
\[
-\langle (A(X) - A(Y))J,X - Y \rangle + \frac{3CR\tau}{4}\|X - Y\|^2 + \frac{3CR}{4\tau}\|B(X) - B(Y)\|^2 \ge c\|X - Y\|^2.
\]
Clearly, we can choose $\tau = \tau(R)$ such that $c - \frac{3CR\tau}{4} \ge \frac{c}{2}$. Therefore, define $\delta\doteq \frac{c}{2}$. Finally by \eqref{algg} we find a constant $\mu = \mu(R) \ge 0$ such that
\[
\|B(X) - B(Y)\|^2 \le - \frac{1}{\mu}\det(B(X) - B(Y)), \forall X,Y\in B_\frac{3R}{2}(0).
\]
This finally concludes the proof of the present Lemma, with $\lambda(R) \doteq \frac{3CR}{4\tau\mu}$.
\end{proof}

\begin{rem}
Notice that inequality \eqref{R} can be interpreted as some sort of "generalized convexity" of the area functional. Indeed, for a function $f \in C^2(\R^{n\times 2})$, the inequality
\[
\langle Df(X)-Df(Y),X - Y\rangle = -\langle (A_f(X)-A_f(Y))J,X - Y\rangle \ge \delta \|X - Y\|^2
\]
is equivalent to convexity. It can be checked that when $n > 1$, the area functional is not convex, hence the previous inequality cannot hold. The previous Lemma shows that adding the term $-\lambda\det(B(X) - B(Y))$ we can nonetheless bound from above the quantity $\|X - Y\|^2$. The key point here is that the determinant is a null Lagrangian and therefore it still allows to prove a regularity result as Proposition \ref{C1gamma}.
\end{rem}

\begin{lemma}\label{pertlem}
Fix $R>0$. Recall that $A_f(X) = Df(X)J$ and $B_f(X) = X^TDf(X)J - f(X)J$. There exists $\varepsilon = \varepsilon(R)$ and $c = c(f,R)>0$ such that if 
\[
\|f - \A\|_{C^2(B_{2R})}\le\varepsilon,
\]
then, for the same constant $\lambda$ of formula \eqref{R},
\begin{equation}\label{genf}
-\langle (A_f(X)-A_f(Y))J,X - Y\rangle - \lambda\det(B_f(X) - B_f(Y)) \ge c\|X - Y\|^2, \text{for every }X,Y \in B_\frac{3R}{2}(0).
\end{equation}
\end{lemma}
\begin{proof}
The proof is by contradiction. Assume we can find a sequence of functions $f_n$, a sequence of numbers $c_n$ and sequences of matrices $X_n$ and $Y_n$ such that
\begin{enumerate}[(i)]
\item $\|f_n - \A\|_{C^2(B_{2R})}\le\frac{1}{n}$;\label{111}
\item $c_n \to 0$;\label{222}
\item $X_n \to X$, $Y_n \to Y$, $\frac{X_n - Y_n}{\|X_n - Y_n\|} \to Z$;\label{333}
\item $-\langle (A_{f_n}(X_n)-A_{f_n}(Y_n))J,X_n - Y_n\rangle - \lambda\det(B_{f_n}(X_n) - B_{f_n}(Y_n)) \le c_n\|X_n - Y_n\|^2$.\label{444}
\end{enumerate}

First, suppose $X \neq Y$. Then, in the limit we find a contradiction with $\eqref{R}$
\[
\delta\|X - Y\|^2 \le -\langle (A(X)-A(Y))J,X - Y\rangle - \lambda\det(B(X) - B(Y)) \le 0.
\]
Now suppose $X = Y$. Define 
\[
T_n \doteq \frac{A_{f_n}(X_n)-A_{f_n}(Y_n)}{\|X_n - Y_n\|} \text{ and } B_n \doteq \frac{B_{f_n}(X_n) - B_{f_n}(Y_n)}{\|X_n - Y_n\|}.
\]
Then, for every $n$, \eqref{444} yields:
\begin{equation}\label{contra}
-\left\langle T_nJ,\frac{X_n - Y_n}{\|X_n - Y_n\|}\right\rangle - \lambda\det(B_n) \le 0.
\end{equation}
We have
\[
T_n = \frac{A_{f_n}(X_n)-A_{f_n}(Y_n)}{\|X_n - Y_n\|} = \frac{\int_0^1DA_{f_n}(tX_n + (1-t)Y_n)[X_n - Y_n]dt}{\|X_n - Y_n\|} \to DA(X)[Z]
\]
and, analogously,
\[
B_n \to DB(X)[Z].
\]
The convergence of $T_n$ and $B_n$ are a direct consequence of \eqref{111}. Consequently, in the limit \eqref{contra} becomes
\begin{equation}\label{contra2}
-\langle DA(X)[Z]J,Z\rangle - \lambda\det(DB(X)[Z])  \le 0.
\end{equation}
Now, by $\eqref{R}$ and for every $n$,
\[
-\langle (A(X_n)-A(Y_n))J,X_n - Y_n\rangle -\lambda\det(B(X_n) - B(Y_n)) \ge \delta\|X_n - Y_n\|^2,
\]
so that, if we divide by $\|X_n - Y_n\|^2$ and pass to the limit, we obtain a contradiction with $\eqref{contra2}$.
\end{proof}

\begin{Teo}\label{pert}
Let $k \ge 2$. For every $R>0$, there exists $\varepsilon = \varepsilon(R)>0$ for which, if $f:\R^{n\times 2}\to \R$ is a function of class $C^k$ with
\[
\|f - \A\|_{C^2(B_{2R})}\le\varepsilon,
\]
then, for every $\U\in W^{1,\infty}(\Omega;\R^{2n + 2})$, $\|D\U\|_{L^\infty}\le R$, such that
\[
D\U(x) \in C_f \text{ for a.e. }x \in \Omega,
\]
it holds $\U \in W^{2,2 + \rho}(\Omega)$, for some positive $\rho$.
\end{Teo}

The proof of the previous Theorem is a consequence of the following result, that in turn is a simple generalization of \cite[Theorem 3]{SAP}.

\begin{prop}\label{C1gamma}
Consider differential inclusions of the following form, for $\mathcal{V} \in W^{1,\infty}_{\loc}(\Omega;\R^{r + m})$,
\begin{equation}\label{diffinc}
D \mathcal{V}(x) \in C = \left\{Y\in\R^{r+ m,2}: Y=
\left(
\begin{array}{c}
X\\
F(X)
\end{array}
\right)
\right\}, \text{ for a.e. } x \in \Omega,
\end{equation}
where $F \in C^k(\R^{r\times 2};\R^{m\times 2})$, $k\ge 1$. Consider moreover the splitting $\mathcal{V} = \left(\begin{array}{c}u \\ v \end{array}\right)$, with $u: \Omega \to \R^r$ and $v:\Omega \to \R^m$. Suppose there exist constants $c_{ab} \in \R$ such that
\begin{equation}\label{reg}
\|X - Y\|^2 \le \sum_{1\le a\le b \le m + r}c_{ab}\det(M^{ab} - N^{ab}),
\end{equation}
for every couple of $M,N \in K$ of the form
\[
M =
\left(
\begin{array}{c}
X\\
F(X)
\end{array}
\right), \quad
N =
\left(
\begin{array}{c}
Y\\
F(Y)
\end{array}
\right).
\]
Then, $u \in W^{2,2+\rho}_{\loc}(\Omega;\R^n)$, for some $\rho >0$.
\end{prop}
\begin{proof}
From now on, we fix open sets $\Omega'\subset\Omega''\subset\Omega$, each with compact closure in the other. For any couple $a,b$ with $1\le a\le b\le m + r$, denote $w_{ab} \doteq\left(\begin{array}{c} \mathcal{V}_a\\ \mathcal{V}_b\end{array}\right)$. Take any nonnegative $\eta\in C^\infty_c(\Omega'')$, $q_{ab} \in \R^2$ constant vectors, and $h \in \R^2$ with $\|h\| \le \frac{\dist(\partial\Omega'',\partial\Omega)}{2}$, and moreover denote, for any function $g:\Omega'\to\R^m$,
\[
g^h(x) = \frac{g(x + h) -g(x)}{\|h\|}.
\]
Since the determinant is a null Lagrangian and $\eta$ has compact support
\begin{align*}
\sum_{ab}c_{ab}\int_{\Omega}\det(D(\eta(x)(w_{ab}^h(x) - q_{ab}))\dx = 0.
\end{align*}
Equation \eqref{sumdet} yields
\begin{align*}
0&=\sum_{ab}c_{ab}\int_{\Omega}\det(D(\eta(x)w_{ab}^h(x) - q_{ab}))\dx = \\
&= \sum_{a,b}c_{ab}\int_{\Omega}\eta^2(x)\det(D w_{ab}^h(x))\dx + \sum_{a,b}c_{ab}\int_{\Omega}\eta(x)\langle\cof^T((w_{ab}^h(x) - q_{ab})\otimes D\eta(x)),D w_{ab}^h(x)\rangle.
\end{align*}
Hence, by \eqref{reg} and our previous computation, we can write
\begin{align*}
\int_\Omega \eta^2(x)\|D u^h(x)\|^2\dx &= \frac{1}{\|h\|^2}\int_\Omega \eta^2(x)\|D (u(x + h) -u(x))\|^2\dx \\
&\le \frac{1}{\|h\|^2}\sum_{ab}c_{ab}\int_{\Omega}\eta^2(x)\det(D (w_{ab}(x + h) - w_{ab}(x)))\dx \\
&= - \sum_{a,b}c_{ab}\int_{\Omega}\eta(x)\langle\cof^T((w_{ab}^h(x)-q_{ab})\otimes D\eta(x)),D w_{ab}^h(x)\rangle\dx\\
&\le  \sum_{a,b}|c_{ab}|\int_{\Omega}\eta(x)\|w_{ab}^h(x) - q_{ab}\|\|D\eta(x)\|\|D w_{ab}^h(x)\|\dx.
\end{align*}
Since $F$ is $C^1$, it is locally Lipschitz. In particular, if $\|u\|_{W^{1,\infty}}\le R$, this implies that, for some constant $c\ge 0$ depending on $R$,
\[
\|D w_{a b}^h(x)\| \le c\|D u^h(x)\|, \text{ a.e. }.
\]
From now on, we will not keep track of the constants, and we will simply denote them by $C$. Continuing our computation, we readily obtain through H\"older's inequality that
\begin{equation}\label{mainineq}
\int_\Omega \eta^2(x)\|D u^h(x)\|^2\dx\le C\sum_{a,b}\int_{\Omega}\|w_{ab}^h(x) - q_{ab}\|^2\|D\eta(x)\|^2\dx.
\end{equation}
Choose $q_{ab} = 0$ for every $a,b$ and $\eta\equiv 1$ on $\Omega'$. Using the fact that $\mathcal{V}$ is Lipschitz, we get
\[
\int_{\Omega'} \|D u^h(x)\|^2\dx\le C(R,\Omega'), \text{ for every sufficiently small }h.
\]
By standard results about Sobolev spaces (see \cite[Proposition 9.3]{BRE}), this implies that $u\in W_{\loc}^{2,2}(\Omega)$. To conclude the proof, we show higher integrability of the Hessian of $u$, namely $D^2u \in L^{2+\rho}$, for some $\rho >0$. To do so, consider again \eqref{mainineq}. This time, consider any square $Q\subset \Omega'$ such that $2Q \subset \Omega'$, where $2Q$ is the square of side $s$ centered at the center of $Q$ but with twice the side. We take $\eta \in C^\infty_c(\sqrt{2}Q)$ with $\eta\equiv 1$ on $Q$, and
\[
\eta\equiv 1 \text{ on $Q$ and } \|D\eta\|(x) \le \frac{C}{s} \text{ on $\sqrt{2}Q$},
\]
for some $C>0$ independent on $x$ and $s$. Then, \eqref{mainineq} becomes
\begin{equation}\label{mainineq2}
\int_Q \|D u^h(x)\|^2\dx\le C\sum_{a,b}\int_{\sqrt{2} Q}\|w_{ab}^h(x) - q_{ab}\|^2\|D\eta(x)\|^2\dx \le \frac{C}{s^2}\sum_{a,b}\int_{\sqrt{2} Q}\|w_{ab}^h(x) - q_{ab}\|^2\dx.
\end{equation}
Now, using \cite[Theorem 3.6]{NPDE}, we can estimate the last term with a Sobolev-type inequality, using $p =2$ and $p^* = 1$, once we have chosen suitably $q_{ab}$:
\begin{align*}
\sum_{a,b}\int_{\sqrt{2} Q}\|w_{ab}^h(x) - q_{ab}\|^2\dx \le C\sum_{a,b}\left(\int_{2Q}\|D w^h_{ab}\|\dx\right)^2.
\end{align*}
Once again, $\|D w^h_{ab}\| \le C\|D u^h\|$ pointwise a.e., where $C$ depends only on the Lipschitz constant of $F$ (that in turn depends only on the Lipschitz constant of $\alpha$). In this way, \eqref{mainineq2} can be rewritten as
\[
\int_Q \|D u^h\|^2\dx \le \frac{C}{s^2}\left(\int_{2Q}\|D u^h\|\dx\right)^2.
\]
Passing to the limit as $h \to 0$, we finally get
\[
\left(\fint_Q\|D^2u\|^2\dx\right)^{\frac{1}{2}} \le C\fint_{2Q}\|D^2u\|\dx.
\]
We can apply Gehring's Lemma as stated, for instance, in \cite[Theorem 1.5]{KIN}, to deduce the higher integrability of the Hessian of our function.
\end{proof}

\subsection{Higher regularity}\label{mor}

By Theorem \ref{pert}, we know that for every $R$, there exists $\varepsilon(R) > 0$ such that
\[
D \U \in C_f \Rightarrow D \U \in W_{\loc}^{2,2 + \rho}(\Omega)
\]
provided that $\|f - \A\|_{C^2(B_{2R})} \le \varepsilon$. In this subsection, we show that, possibly taking a smaller $\varepsilon$, if $f \in C^k$, for $k \ge 2$, then $\U \in C^{k - 1}$. The procedure here is quite stardard (see, for instance, \cite[Corollary]{SAP}) and we describe it for the reader's convenience. To show the improvement of regularity, we exploit the results of \cite{mora,morb}. Suppose that
\[
f \in C^k(\R^{n\times 2},\R),\quad k \ge 2
\]
satisfies the following {Legendre-Hadamard} condition (briefly, LH), i.e. there exists a constant $\mu > 0$ such that
\begin{equation}\label{LH}
D^2f(X)[Y,Y] \ge \mu\|Y\|^2,\; \forall X,Y \in \R^{n\times 2}, \rank(Y) = 1,
\end{equation}
where
\[
D^2f(X)[Y,Y] \doteq \frac{d^2}{dt^2}|_{t = 0} f(X + tY).
\]
Then, applying \cite[Theorem 6.2.5]{morb}, we infer that the $W^{2,2 + \rho}$ solutions of
\[
\dv(Df(D u)) = 0
\]
belong to $C_{\loc}^{k-1,\alpha}$, for some $\alpha$ depending on $\rho$. In order to apply \cite[Theorem 6.2.5]{morb}, we need to prove that functionals close to the area satisfies the LH condition. In Lemma \ref{areaLH} we prove that the area satisfy a local LH condition, and in Lemma \ref{genLH} we extend this to functions close to the area. To apply Morrey's \cite[Theorem 6.2.5]{morb}, we need to prove a global LH condition for these functionals. Nevertheless, since we are just interested in Lipschitz solution of constant $R > 0$, it will be sufficient to prove that there exists an extension of the function $f$ under consideration to the whole $\R^{n\times 2}$ that satisfies the LH condition. This extension is the content of Lemma \ref{EXT}.

\begin{lemma}\label{areaLH}
For every $R >0$, there exists a constant $\tau(R)>0$ such that
\[
D^2\A(X)[Y,Y] \ge \tau\|Y\|^2,\; \forall X,Y \in \R^{n\times 2},X \in B_{\frac{3R}{2}}(0), \rank(Y) = 1.
\]
\end{lemma} 

\begin{proof}
Fix $X \in \R^{n\times 2}$, $\|X\| \le R$, and $Y \in \R^{n\times 2}$ with $\|Y\| = 1$ and $\rank(Y) = 1$. Define the function $$g(t) \doteq \A(X + tY).$$
The thesis is equivalent to
\[
g''(0) \ge \tau(R).
\]
Since $\rank(Y) = 1$
\[
g(t) = \sqrt{1 + \|X + tY\|^2 + \sum_{a,b}(\det(X^{ab}) + t\langle X^{ab},\cof^T(Y^{ab})\rangle)^2 }.
\]
Therefore,
\[
g'(t) = \frac{s(t)}{g(t)},
\]
where
\[
s(t) = \langle X + tY,Y\rangle + \sum_{a,b}(\det(X^{ab}) + t\langle X^{ab},\cof^T(Y^{ab})\rangle)\langle X^{ab},\cof^T(Y^{ab})\rangle.
\]
This implies
\[
g''(t) = \frac{s'(t)}{g(t)} - \frac{s(t)g'(t)}{g^2(t)} =  \frac{s'(t)}{g(t)} - \frac{s^2(t)}{g^3(t)} = \frac{s'(t)g^2(t) - s^2(t)}{g^3(t)}.
\]
Finally:
\[
g''(0) =  \frac{s'(0)g^2(0) - s^2(0)}{g^3(0)}.
\]
We will now show that $s'(0)g^2(0) - s^2(0) \ge 1$, and this concludes the proof. To simplify the notation, define 
\begin{align*}
&A \doteq \sum_{a,b} \langle X^{ab},\cof^T(Y^{ab})\rangle^2,\\
&B \doteq \sum_{a,b}\det(X^{ab})\langle X^{ab},\cof^T(Y^{ab})\rangle.
\end{align*}
Recall that we are assuming $\|Y\| = 1$, and that $ \sum_{a,b}(\det(X^{ab}))^2 = \det(X^TX)$. Therefore:
\begin{equation}\label{end}
s'(0)g^2(0) - s^2(0) = 1 + \|X\|^2 + \det(X^TX) + A + A\|X\|^2 + A\det(X^TX) -  (\langle X,Y\rangle + B)^2,
\end{equation}
and
\begin{align*}
&\|X\|^2 + \det(X^TX) + A\|X\|^2 + A\det(X^TX) -  (\langle X,Y\rangle + B)^2\\
&= (\|X\|^2 - \langle X,Y\rangle^2)  + (A\det(X^TX) - B^2) + (\det(X^TX) + A\|X\|^2 -  2\langle X,Y\rangle B).
\end{align*}
We claim that the terms in brackets of the previous expression are all nonnegative. This would conclude the proof, since then, considering \eqref{end}
\[
s'(0)g^2(0) - s^2(0) \ge 1 + A
\]
and $A \ge 0$. Let us prove the claim. First, we need to show that $$\|X\|^2 - \langle X,Y\rangle^2 \ge 0.$$ Cauchy-Schwartz inequality and the fact that $\|Y\| = 1$ imply
\[
\|X\|^2 - \langle X,Y\rangle^2 \ge \|X\|^2 - \|X\|^2\|Y\|^2 =  \|X\|^2 - \|X\|^2 = 0.
\]
The second inequality we need is
\[
B^2 \le A\det(X^TX).
\]
By the definition of $B$ and applying again Cauchy-Schwartz inequality:
\[
B^2 = \left(\sum_{a,b}\det(X^{ab})\langle X^{ab},\cof^T(Y^{ab})\rangle\right)^2 \le \sum_{a,b}\det(X^{ab})^2\sum_{a,b}\langle X^{ab},\cof^T(Y^{ab})\rangle^2 = A\det(X^TX).
\]
Finally, we prove that
\[
2\langle X,Y\rangle B \le A\|X\|^2 + \det(X^TX).
\]
By Cauchy-Schwartz and Young inequality:
\begin{align*}
2\langle X,Y\rangle B &= 2\langle X,Y\rangle\sum_{a,b}(\det(X^{ab})\langle X^{ab},\cof^T(Y^{ab})\rangle) \\
&\le 2|\langle X,Y\rangle|\sqrt{\sum_{a,b}\det(X^{ab})^2}\sqrt{\sum_{a,b}\langle X^{ab},\cof^T(Y^{ab})\rangle^2}\\
& = 2|\langle X,Y\rangle|\det(X^TX)^{\frac{1}{2}}A^{\frac{1}{2}}\\
& \le A|\langle X,Y\rangle|^2 + \det(X^TX) \le A\|X\|^2 + \det(X^TX).
\end{align*}
\end{proof}

\begin{lemma}\label{genLH}
For every $R> 0$, there exists $\varepsilon'(R) >0$ such that, if $f \in C^2(\R^{n\times 2})$ and
\[
\|f - \A\|_{C^2(B_{2R})} \le \varepsilon'(R),
\]
then there exists a constant $\tau' = \tau'(R)$ such that
\[
D^2f(X)[Y,Y] \ge \tau'\|Y\|^2, \forall X, Y \in \R^{n\times 2}, \|X\| \le \frac{3R}{2}, \rank(Y) = 1.
\]
\end{lemma}
\begin{proof}
Suppose by contradiction that the thesis is false. Then, we can find a sequence of functions $f_n$, a sequence of positive numbers $c_n$ and sequences of matrices $X_n$, $Y_n$ such that:
\begin{enumerate}[(i)]
\item $\|f_n - \A\|_{C^2(B_{2R})} \le \frac{1}{n}$;
\item $c_n \to 0$;
\item $X_n \to X$;
\item $\|Y_n\|= 1$, $\rank(Y_n) = 1$, and $Y_n \to Y \in \R^{n\times 2}, \|Y\| = 1, \rank(Y) = 1$
\item $D^2f_n(X_n)[Y_n,Y_n] \le c_n$.\label{5}
\end{enumerate}
Passing to the limit in $\eqref{5}$, we immediately get a contradiction with Lemma \ref{areaLH}.
\end{proof}

In order to prove the next lemma we need to introduce a new:

\begin{Def}
Let $\mu \ge 0$. The function $h: \R^{n\times 2} \to \R$ is $\mu$-rank-one convex if and only if for every $X,Y \in \R^{n\times 2}$, $\rank(Y) = 1$,
\[
\phi(t) \doteq h(X + tY)
\]
is a uniformly convex function with constant $\mu$, i.e.
\[
\phi(at_1 + bt_2) \le t_1\phi(a) + t_2\phi(b) - t_1t_2\mu|a - b|^2, \; \forall a,b,t_1,t_2\in \R, t_1 + t_2 = 1, t_1,t_2\ge 0.
\]
If $\mu = 0$, the function $h$ is simply called {rank-one convex}.
\end{Def}

It is not difficult to see that if $h \in C^2(\R^{n\times 2})$, then $h$ is $\mu$-rank-one convex if and only if it satisfies the LH condition with constant $\mu$ (i.e., \eqref{LH} holds). Therefore we will say that a $C^2$ function $h$ is $\mu$- rank-one convex in $B_{r}(0)$ for some $r > 0$ if and only if \eqref{LH} holds for every $X \in B_r(0) \subset \R^{n\times 2}$ and for every $Y \in \R^{n\times 2}$ with $\rank(Y) = 1$.

\begin{lemma}\label{EXT}
Let $f\in C^k(B_{2R})$, $k \ge 2$, be a $\mu$-rank-one convex function on $B_{2R}$. Then, there exists a function $F$ such that
\begin{itemize}
\item $F = f$ on $B_{\frac{3R}{2}}$;
\item $F \in C^k(\R^{n\times 2})$;
\item $F$ is $\frac{\mu}{2}$ rank-one convex. 
\end{itemize}
\end{lemma}
\begin{proof}
Choose any $R_1 \in \left(\frac{3R}{2},2R\right)$. Moreover, define $f'(X) \doteq f(X) - \frac{3\mu\|X\|^2}{4}$. Notice that, by our hypothesis, $f'$ is still rank-one convex on $B_{2R}(0)$. Apply \cite[Lemma 2.3]{SMVS} to find a rank one convex function $F': \R^{n\times 2} \to \R$ such that $F'$ coincides with $f'$ on $B_{R_1}$. The function $$F''(X)\doteq F'(X) + \frac{3\mu\|X\|^2}{4}$$ is $\frac{3\mu}{4}$- rank-one convex on the whole $\R^{n\times 2}$ and on $B_{R_1}$ it coincides with $f(X)$. We take any family of mollifiers $\rho_\varepsilon$ on $\R^{n\times 2}$ with $\spt(\rho_\varepsilon) \subset B_{\varepsilon}(0)$ and $\rho_\varepsilon(X) \ge 0$ for every $X \in \R^{n\times 2}$, and define
\[
F_\varepsilon(X) \doteq (F''\star\rho_\varepsilon)(X), \forall X \in \R^{n\times 2}.
\]
The convolution is well defined since rank-one convexity implies that $F''$ is locally Lipschitz. Through a direct computation, it is easy to see that $F_\varepsilon$ is still $\frac{3\mu}{4}$-rank one convex. Consider any $R_2 \in \left(\frac{3R}{2},R_1\right)$ and take a function $\eta \in C_c^{\infty}(\R^{n\times 2})$ such that $0\le\eta(X)\le 1,\forall X,$ $\eta \equiv 1$ on $B_{R_2 + \delta}$ and $\eta \equiv 0$ on $B^c_{R_1- \delta}$, with $0 <\delta \doteq \frac{R_1 - R_2}{10}$. Next, define
\[
G_\varepsilon(X) \doteq \eta(X) F''(X) + (1-\eta(X))F_{\varepsilon}(X).
\]
We claim that there exists $\varepsilon>0$ such that $G_\varepsilon(X)$ has the desired properties. Indeed, for every $\varepsilon > 0$, $G_\varepsilon$ is a $C^k(\R^{n\times 2})$ function that coincides with $F''$ and therefore $f$ on $B_{\frac{3R}{2}}$. Moreover, by the properties of the support of $\eta$ and the $\frac{3\mu}{4}$-rank one convexity of $F''$ and $F_\varepsilon$, it holds
\[
D^2G_\varepsilon(X)[Y,Y] \ge \frac{3\mu}{4}\|Y\|^2
\]
for every $\varepsilon >0$, $Y \in \R^{n\times 2}$ with $\rank(Y) = 1$ and $X \in \mathcal{B}\doteq \bar{B}_{R_2 + \frac{\delta}{2}}\cup B^c_{R_1 - \frac{\delta}{2}}$. Therefore, to conclude the proof, we need to show that for $\varepsilon> 0$ sufficiently small,
\[
D^2G_\varepsilon(X)[Y,Y] \ge \frac{\mu}{2}\|Y\|^2, \text{ for } X \in \mathcal{B}^c. 
\]
Take $\varepsilon < \frac{R_1 - R_2}{100}$. In this case, we see that for every $X \in \mathcal{B}^c$
\begin{equation}\label{moll}
DF_\varepsilon(X) = (DF''\star\rho_\varepsilon)(X) \text{ and } D^2F_\varepsilon(X) = (D^2F''\star\rho_\varepsilon)(X),
\end{equation}
since $F''$ coincides with the $C^k$ ($k\ge 2$) function $f$ on $B_{R_1}$. We obtain
\begin{align*}
D^2G_\varepsilon &= F''D^2\eta + (D\eta\otimes DF'' + DF''\otimes D\eta) +\eta D^2F'' \\
&- F_\varepsilon D^2\eta - (D\eta\otimes DF_\varepsilon + DF_\varepsilon\otimes D\eta) + (1 - \eta) D^2F_\varepsilon.
\end{align*}
Define
\begin{align*}
V_\varepsilon &\doteq F''D^2\eta + (D\eta\otimes DF'' + DF''\otimes D\eta) \\
&- F_\varepsilon D^2\eta - (D\eta\otimes DF_\varepsilon + DF_\varepsilon\otimes D\eta).
\end{align*}
For every tensor $W = (W_{abcd}), a,c \in \{1,\dots n\}, b,d \in \{1,2\}$, denote with
\[
W[Y,Y] \doteq \sum_{a,b,c,d}W_{abcd}y_{ab}y_{cd},\quad \forall Y = (y_{ij}) \in \R^{n\times 2}.
\]
Exploiting \eqref{moll} and the regularity of $F''$, we see that there exists a constant $C> 0$ independent of $X$ such that
\[
|V_\varepsilon(X)[Y,Y]| \le C\varepsilon\|Y\|^2,
\]
for every $X \in \mathcal{B}^c$ and every $Y \in \R^{n\times 2}$ (non necessarily with $\rank(Y) = 1$). We can choose any number $ 0 < \varepsilon \le \frac{\mu}{4C}$. Let it be $\varepsilon_0$, and call $F(X) \doteq G_{\varepsilon_0}(X)$. $F$ has the three properties listed in the statement of the Lemma.
\end{proof}

We can summarize the result of this section in the following

\begin{Teo}\label{FINALE}
For every $R > 0$, there exists $\alpha = \alpha(R) > 0$ such that, if $f$ is a $C^k(\R^{{2n + 2}\times 2})$ function, $k \ge 2$, with the property that
\begin{equation}\label{near}
\|f - \A\|_{C^2(B_{2R}(0))} \le \alpha,
\end{equation}
and $\U: \Omega \to \R^{2n+ 2}$ is a Lipschitz solution of
\begin{equation}\label{incc}
D \U(x) \in C_f, \text{ for a.e. }x\in \Omega
\end{equation}
with
\[
\|D \U\|_{\infty} \le R,
\]
then
$\U \in C^{k-1,\rho}(\Omega)$, for some positive $\rho > 0$.
\end{Teo}

\begin{proof}
Fix $R > 0$. Choose $\alpha(R) \doteq \min\{\varepsilon(R),\varepsilon'(R)\}$, where $\varepsilon$ and $\varepsilon'$ are defined in Lemma \ref{pertlem} and Lemma \ref{genLH} respectively. Take any $f$ satisfying $\eqref{near}$ and a $R$-Lipschitz $\U$ satisfying \eqref{incc}. By our choice of $\alpha$, $\U$ belongs to $W_{\loc}^{2,2 + \rho}(\Omega)$ by Theorem \ref{pert}. Again, by the choice of $\alpha$, by Lemma \ref{genLH} we have that $f$ satisfies the LH condition in $B_{2R}$. Using Lemma \ref{EXT}, we can consider $F \in C^k(\R^{n\times 2})$ that extends $f$ outside $B_{\frac{3R}{2}}$ and that satisfies the LH condition on the whole $\R^{n\times 2}$. Since $\|D\U\|_{\infty} \le R$,
\[
\dv(DF(D u)) = \dv(Df(D u)) = 0, \text{a.e. in } \Omega.
\]
$\U$ has the desired regularity by \cite[Theorem 6.2.5]{morb}, as described at the beginning of this subsection.
\end{proof}

\section{Irregular critical points for inner variations}\label{IRREG}

The purpose of this section is to show the following:
\begin{Teo}\label{IRR}
Let $\Omega$ be an open and bounded subset of $\R^2$. There exists a map $\psi \in W^{1,p}(\Omega,\R^2)$ for some $p > 2$ that solves
\[
\curl(B(D\psi)) = 0,
\]
and such that for every open $\mathcal{V} \subset \Omega$, $\psi$ is not $C^1(\mathcal{V})$.
\end{Teo}
The proof of this result is achieved by combining a simple Linear Algebra lemma, Lemma \ref{algebra}, with the counterexample constructed in \cite[Example 4.41]{KIRK}. First, let us define
\[
H_1 \doteq \left\{X \in \R^{2\times 2}: X =
\left(
\begin{array}{cc}
a&b\\
b&-a
\end{array}
\right)\right\}
\]
and
\[
H_2 \doteq \left\{X \in \R^{2\times 2}: X =
\left(
\begin{array}{cc}
a&-b\\
b&a
\end{array}
\right)\right\}.
\]

\begin{lemma}\label{algebra}
For every $X \in H_1\cup H_2$, we have
\[
A(X) = XJ
\]
and
\[
B(X) = J
\]
\end{lemma}
\begin{proof}
Let us consider the matrix
\[
X = \left(
\begin{array}{cc}
a&\alpha b\\
b&\beta a
\end{array}
\right),
\]
with $\alpha= \pm 1$, $\alpha\beta = -1$, $a,b \in \R$. Clearly, every matrix in $H_1\cup H_2$ is of this form. We have
\[
\|X\|^2 = 2(a^2  + b^2),\quad \det(X)^2 = (a^2 + b^2)^2,
\]
hence
\[
\A(X) = 1 + a^2 + b^2.
\]
Moreover,
\[
\cof(X)^T = \left(
\begin{array}{cc}
\beta a&-b\\
-\alpha b& a
\end{array}
\right),
\]
thus
\[
X + \det(X)\cof(X)^T =
\left(
\begin{array}{cc}
a&\alpha b\\
b&\beta a
\end{array}
\right)
+
(\beta a^2 - \alpha b^2)\left(
\begin{array}{cc}
\beta a&-b\\
-\alpha b& a
\end{array}
\right)=
\A(X)\left(
\begin{array}{cc}
a&\alpha b\\
b&\beta a
\end{array}
\right) =\A(X) X.
\]
Therefore, $A(X) = XJ$. We now prove that $B(X) = -J$. To do so, we compute
\[
X^TX = 
\left(
\begin{array}{cc}
a& b\\
\alpha b&\beta a
\end{array}
\right)
\left(
\begin{array}{cc}
a&\alpha b\\
b&\beta a
\end{array}
\right) = (a^2 + b^2) \id = \frac{\|X\|^2}{2}\id.
\]
Hence
\[
\A(X)B(X) = -\left(1 +\frac{\|X\|^2}{2}\right)J = -\A(X)J.
\]
This concludes the lemma.
\end{proof}
In \cite[Example 4.41]{KIRK} it is shown that there exists a Sobolev map $\psi\in W^{1,p}(\Omega,\R^2), p > 2,$ such that $D\psi$ belongs, at almost every point of $\Omega$, to $H_1\cup H_2$, and moreover
\[
|\{x \in \Omega: D \psi(x) = 0\}| > 0
\]
but $\psi$ is non-constant. By Lemma \ref{algebra}, we immediately deduce that this function $\psi$ solves
\[
\curl(B(D \psi(x))) = \curl(-J) = 0,
\]
hence it is a solution to the inner variations equations for the area function. We want to construct such a $\psi$ by using the same methods of \cite[Example 4.41]{KIRK}, but we moreover want to construct it in such a way that for every open subset $\mathcal{V}\subset \Omega$
\[
|\{y \in \mathcal{V}: D \psi(y) = 0\}| > 0
\]
but $\psi$ is non-constant in $\mathcal{V}$. In this way, we would deduce that $\psi$ cannot be $C^1$ on any open set. In fact, suppose by contradiction that there exists a connected open set $\mathcal{V}$ such that $\psi \in C^1(\mathcal{V})$. Let $\mathcal{W} \subset \mathcal{V}$ be an open, compactly contained subset of $\mathcal{V}$. Since $H_1,H_2$ are closed, we obtain that
\[
A_i \doteq \{y\in \overline{\mathcal{W}}: D \psi(y) \in H_i\}
\]
are closed sets, contained in $\mathcal{W}$, for $i = 1,2$, and that moreover
\[
\overline{\mathcal{W}} = A_1\cup A_2.
\]
There are two cases: $A_1$ does not contain any ball or there exists $B_r(y) \subset A_1$. If $\intt(A_1) = \emptyset$, then $A_2$ is dense in $\overline{\mathcal{W}}$. Since it is also closed, then $\overline{\mathcal{W}} = A_2$. In particular, on the open set $\mathcal{W}$, one has $D\psi \in H_2$. This implies that $\psi$ is harmonic and smooth. It is well-known that for a non-constant harmonic function $\psi$
\[
|\{y \in \mathcal{W}: D\psi(y) = 0\}| =0,
\]
which is a contradiction with $|\{y \in \mathcal{W}: D \psi(y) = 0\}| > 0$. Therefore, we are left with the case $B_r(y) \subset A_1$. But then, exactly the same reasoning applied with $B_r(y)$ instead of $\mathcal{W}$ yields the same contradiction.
\\
\\
This discussion motivates the fact that, in order to conclude that we can find a solution that is not $C^1$ in any open set of $\Omega$, we need the following
\begin{lemma}\label{PSI}
There exists an open set $\Omega$ and a $W^{1,p}$, $p > 2$, map $\psi: \Omega \to \R^2$ with the property for every open set $\mathcal{V} \subset \Omega$,
\begin{itemize}
\item $\psi$ is non-constant on $\mathcal{V}$;
\item $|\mathcal{V}\cap\{y \in \Omega: D\psi(y) = 0\}| > 0$.
\end{itemize}
\end{lemma}
To prove Lemma \ref{PSI}, it is sufficient to show the following:
\begin{lemma}\label{SUFF}
There exists a Lipschitz map $f:B_1(0)\subset \R^2 \to \R^2$ with the following properties:
\begin{itemize}
\item $Df(x) \in \{A_1,\dots,A_5\}$ for five $2\times 2$ matrices $A_1,\dots, A_5$ (explicitely written in \cite[Example 4.41]{KIRK}), for a.e. $x \in B_1(0)$;
\item If $\mathcal{A}_i \doteq \{x \in B_1(0): Du(x) = A_i\}$, then for every open subset of $B_1(0)$, $B$, it holds
\[
|B\cap \mathcal{A}_i| \neq 0,\quad \forall i =1,\dots, 5.
\]
\end{itemize}
\end{lemma}

If Lemma \ref{PSI} holds, then the previous discussion constitutes the proof of Theorem \ref{IRR}. Let us now explain how Lemma \ref{SUFF} implies Lemma \ref{PSI}.  

\begin{proof}[Proof of Lemma \ref{PSI}]
This proof is exactly the same described in \cite[Example 4.41]{KIRK}, and we report it here for the reader's convenience. Suppose a map $f$ as the one of Lemma \ref{SUFF} exists. We can define the mapping $\psi$ as in \cite[Example 4.41]{KIRK}, i.e. $\psi(x) \doteq f(F^{-1}(x))$, where $F: \R^2\to\R^2$ is a suitable $W^{1,p}$, $p > 2$ quasiregular homeomorphism. Since we do not need to explicitely introduce quasiregular maps or Beltrami equations, we will not enter in the details of this theory. We refer the interested reader to the references given in \cite[Example 4.41]{KIRK}. The open set $\Omega$ is $\Omega \doteq F(B_1(0))$. The map $F$ satisfies a suitable Beltrami equation, introduced in such a way that for a.e. $y \in F(\mathcal{A}_1\cup \mathcal{A}_2)$, we have $D\psi(y) \in H_1$, while for a.e. $y \in F(\mathcal{A}_2\cup \mathcal{A}_3\cup \mathcal{A}_4 \cup \mathcal{A}_5)$, we have $D\psi(y) \in H_2$. Moreover, by the computations of \cite[Example 4.41]{KIRK} (in particular, by the equation following (4.10)), we find that 
\begin{equation}\label{nonc}
 y \in F(\mathcal{A}_1) \Rightarrow D \psi(y) \neq 0.
\end{equation}
Now let $\mathcal{V} \subset \Omega$ be open. We want to show that $\psi$ is non-constant on $\mathcal{V}$ and
\[
|\{x \in \mathcal{V}: D\psi(x) = 0\}| > 0.
\]
We claim
\begin{equation}\label{claim}
|\mathcal{V}\cap F(\mathcal{A}_i)|> 0, \forall i \in \{1,\dots 5\}.
\end{equation}
Indeed, if for some $i$ we had
\[
|\mathcal{V}\cap F(\mathcal{A}_i)| = 0,
\]
then, making repeated use of the fact that $F$ is bijective,
\begin{equation}\label{contrinn}
0 = |F(F^{-1}(\mathcal{V}))\cap F(\mathcal{A}_i)| = |F(F^{-1}(\mathcal{V})\cap \mathcal{A}_i)|.
\end{equation}
From \cite[Corollary 3.7.6]{IWABOO} we see that $F$ has the $N^{-1}$ property, i.e. for every Borel set $A$,
\[
|A| = 0 \Rightarrow |F^{-1}(A)| = 0.
\]
With this, we can infer from \eqref{contrinn} that there exists $i$ such that
\[
|F^{-1}(\mathcal{V})\cap \mathcal{A}_i| = 0.
\]
Since $F^{-1}$ is an open mapping, then $F^{-1}(\mathcal{V})$ is an open set, hence the previous equality is in contradiction with the properties of the map $f$. Using \eqref{claim}, we immediately see that, on $\mathcal{V}$, $\psi$ cannot be constant since, as noted in \eqref{nonc}, $D\psi \neq 0$ on $F(\mathcal{A}_1)$. On the other hand, $y \in F(\mathcal{A}_2) \Rightarrow D\psi(y) = 0$. This implies, again by \eqref{claim}, that $D\psi(y) = 0$ on a set of positive measure inside $\mathcal{V}$, but $\psi$ is not constant on $\mathcal{V}$.
\end{proof}

In the next and final subsection we will show Lemma \ref{SUFF}.

\subsection{Convex integration: proof of Lemma \ref{SUFF}}

To prove Lemma \ref{SUFF}, we use the Baire Category arguments of \cite{KIRK}. First, we need to recall the following:

\begin{Def}
Let $\mathcal{U}\subset\R^{n\times m}$ be bounded and $K\subset \R^{n \times m}$ be closed. We say that gradients in $\mathcal{U}$ are stable only near $K$ if for every $\varepsilon >0$, one can find $\delta =\delta(\varepsilon) > 0$ such that, if $A \in \mathcal{U}$ and $\dist(A,K) > \varepsilon$, then there exists a piecewise affine map $\varphi \in \Lip(\R^n,\R^m)$ with bounded support such that
\begin{itemize}
\item $D \varphi(x) + A \in \mathcal{U}$ for a.e. $x \in \R^n$;
\item $\displaystyle \int\|D \varphi\|\dx \ge \delta|\spt(\varphi)|.$
\end{itemize}
\end{Def}

The reason why this definition is useful is given by the following result, see \cite[Proposition 3.17, Corollary 3.18]{KIRK}. Let
\[
\mathcal{P}\doteq\{u\in \Lip(\Omega,\R^n): u \text{ piecewise affine, } Du(x) \in \mathcal{U} \text{ a.e. in } \Omega\}
\]
and define the complete metric space
\begin{equation}\label{Xmet}
X \doteq \overline{\mathcal{P}}^{\|\cdot\|_{L^\infty}}.
\end{equation}

\begin{prop}
Let the gradients of $\mathcal{U}$ be stable only near a closed set $K$. Then the typical map $u \in X$ has the property
\[
Du \in K \text{ a.e..}
\]
\end{prop}

We now show Lemma \ref{SUFF}, but first we need to explain how to obtain the matrices $\{A_1,\dots,A_5\}$ in the statement of the Lemma. These matrices are obtained from another set of five symmetric matrices $K\doteq \{P_{F_0},P_{B_0},P_{R_0},P_{L_0},P_{H_0}\}$ simply by considering $M(K - P_{F_0}) = \{A_1,\dots, A_5\}$, where $M$ is a suitable $2\times 2$ matrix. The importance of the set $K$, found by Kirchheim and D. Preiss in \cite[Construction 4.38]{KIRK}, is due to the fact that it is the first example in the literature of a set of five "non-rigid" matrices, i.e. such that there exists a non-affine map $u \in \Lip(B_1(0),\R^2)$ that fulfills $$Du(x) \in K$$ for a.e. $x \in B_1(0)$. The strategy they use is to find an open subset $\mathcal{U}$ of $\Sym(2)$ such that gradients of $\mathcal{U}$ are stable only near $K$, see \cite[Construction 4.38]{KIRK}. We can now start the:

\begin{proof}[Proof of Lemma \ref{SUFF}]
Following the previous notation we consider $K = \{P_{F_0},P_{B_0},P_{R_0},P_{L_0},P_{H_0}\}$ and $\mathcal{U}$ be the open subset of $\Sym(2)$ found by Kirchheim and D. Preiss in \cite[Construction 4.38]{KIRK}. We consider $X$ defined as in \eqref{Xmet}. Now enumerate the points with rational coordinates in $B_1(0)$, $\{q_i\}_{i \in \N}$, and define the sets
\[
X_{q_i,r,j}\doteq\{u \in X: u\text{ is affine in } B_{r}(q_i)\}.
\]
for rational $0 < r < \dist(q_i,\partial B_1(0))$ and $1\le j \le 5$. We aim to show $Y\doteq \bigcup_{i,r,j}X_{q_i,r,j}$ is meager. If this is the case, then $Z\doteq Y^c\cap \{u \in X: Du(x) \in K, \text{ for a.e. }x\in \Omega\}$ is residual in $X$. Baire Theorem tells us that it is non-empty, and obviously for any $u \in Z$, one has
\[
Du(x) \in K = \{P_{F_0},P_{B_0},P_{R_0},P_{L_0},P_{H_0}\}, \text{ for a.e. }x \in \Omega.
\]
Considering $f(x) \doteq M(u(x) - P_{F_0}x)$, where $M$ was introduced before the proof of the present Lemma, we get
\[
Df(x) \in \{A_1,A_2,A_3,A_4,A_5\}, \text{ a.e.}.
\]
Moreover, for every $1 \le j \le 5$, $q\in \mathbb{Q}^2\cap \Omega$, rational radius $0< r< \dist(x,\partial \Omega)$, 
\begin{equation}\label{balls}
|\mathcal{A}_j\cap B_{r}(q)| > 0.
\end{equation}
Indeed, if $|\mathcal{A}_j\cap B_{r}(q)| = 0$, by the rigidity for the four gradients problem, see \cite[Theorem 4.33]{KIRK}, we get that $f$ is necessarily affine on $B_r(q)$, against the definition of $Z$. Since \eqref{balls} is clearly equivalent to
\[
|\mathcal{A}_j\cap \mathcal{V}| > 0
\]
for every open subset $\mathcal{V}\subset \Omega$ and $1\le j \le 5$, we would then conclude the proof. In order to show that $Y$ is meager, we prove that $X_{q_i,r,j}$ are closed sets with empty interior. The closedness inside the complete metric space $X$ is straighforward, since a sequence of affine functions converging in $L^\infty$ need to converge to an affine function. Now suppose by contradiction that for some $i,r,j$, $X_{q_i,r,j}$ has non-empty interior. In particular, we suppose we have that for some $\alpha > 0$ and $u \in X$, 
\[
\{v\in X: \|u - v\|_{\infty} < \alpha\} \subset X_{q_i,r,j}.
\]
Since $u \in X$, we can pick a function $\bar u \in \mathcal{P}$ such that $\|\bar u - u\| \leq \frac{\alpha}{4}$ and $D\bar{u} \in \mathcal{U}$. We also know, by assumption, that $\bar{u}$ is affine on $B_r(q_i)$, say $\bar u = Ax + b$ on $B_r(q_i)$ with $A \in \mathcal{U}$. Since $A \in \mathcal{U}$, that is an open subset of $\Sym(2)$, as follows by the construction of \cite{KIRK}, then we can easily find two matrices $B$ and $C$ in $\mathcal{U}$ such that $\rank(B - C) = 1$ and $\frac{B + C}{2} = A$. For instance, one can take
\[
B\doteq A + \lambda E_{11},\quad C\doteq  A -\lambda E_{11},
\]
where $\lambda> 0$ is a sufficiently small parameter and $$E_{11} \doteq \left(\begin{array}{cc}1 & 0\\ 0&0\end{array}\right).$$ By \cite[Proposition 3.4]{KIRK}, recalled below, for every $\varepsilon > 0$ we can find a Lipschitz and piecewise affine map $w: B_{\frac{r}{2}}(q_i)\to \R^2$ with
\begin{itemize}
\item $Dw(x) \in \mathcal{U}$ a.e.;
\item $w(x) = Ax$ on $\partial B_{r/2}(q_i)$;
\item $\|w - A\|_\infty \le \varepsilon$.
\end{itemize}
Of course, if we traslate $w$ with $\bar w\doteq w + b$, we have
\begin{itemize}
\item $D\bar{w}(x) \in \mathcal{U}$ a.e.;
\item $\bar{w}(x) = \bar u(x)$ on $\partial B_{r/2}(q_i)$;
\item $\|\bar{w} - \bar{u}\|_{L^{\infty}(B_{r/2}(q_i))} \le \varepsilon$.
\end{itemize}
Moreover, the same proposition yields the following property
\[
|\{x \in B_{r/2}(q_i): D \bar w(x) = B\}|\ge \frac{(1 - \varepsilon)}{2}|B_{r/2}(q_i)|
\]
and
\[
|\{x \in B_{r/2}(q_i): D \bar w(x) = C\}|\ge \frac{(1 - \varepsilon)}{2}|B_{r/2}(q_i)|.
\]
In particular, this implies that $\bar w$ cannot be affine on $B_{r/2}(q_i)$. We finally get a contradiction, because the map
\[
z(x) \doteq
\begin{cases}
\bar u(x), & \text{ if } x \in \Omega \setminus B_{r/2}(q_i)\\
\bar w(x), & \text{ if } x \in B_{r/2}(q_i)
\end{cases}
\]
is piecewise affine, Lipschitz, $\|z - \bar u\|_{L^\infty(\Omega} \le \varepsilon$ and $Dz(x) \in \mathcal{U}$, for a.e. $x \in \Omega$. If $\varepsilon < \frac{\alpha}{4}$, then we would obtain that $z$ is affine on $B_{r/2}(q_i)$, against the construction of $\bar w$. This concludes the proof.
\end{proof}

We conclude by recalling here \cite[Proposition 3.4]{KIRK} for the reader's convenience:

\begin{prop}\label{tocite}
Let $A,B,C \in \Sym(n)$, with $\rank(B - C) = 1$, and $A = tB + (1 - t)C$, for some $t \in [0,1]$. Let also $\Omega \subset \R^n$ be a fixed open domain. Then, for every $\varepsilon > 0$, one can find a Lipschitz piecewise affine map $f: \Omega\to \R^n$ such that
\begin{itemize}
\item $f(x) = Ax$ on $\partial\Omega$ and $\|f - A\|_{\infty} \le \varepsilon$;
\item $D f(x) \in \Sym(n)\cap B_{\varepsilon}([B,C])$;
\item $|\{x \in \Omega: D f(x) = B\}|\ge (1 - \varepsilon)t|\Omega|$ and $|\{x \in \Omega: D f(x) = C\}|\ge (1 - \varepsilon)(1 - t)|\Omega|$.
\end{itemize}
\end{prop}

\bibliographystyle{plain}
\bibliography{Area}
\end{document}